\documentclass[12pt,oneside,reqno]{amsart}

\usepackage{amsmath,amssymb,amsthm,textcomp}
\usepackage{dsfont}
\usepackage{amsfonts,graphicx}
\usepackage[mathscr]{eucal}
\usepackage{color}
\usepackage{csquotes}
\usepackage{relsize}
\usepackage{hyperref}
\usepackage[backend=bibtex,%
firstinits=true,%
doi=false,%
isbn=true,%
url=false,%
maxnames=99]{biblatex}%

\AtEveryBibitem{\clearfield{issn}}
\AtEveryCitekey{\clearfield{issn}}
\numberwithin{equation}{section}
\DeclareNameAlias{sortname}{last-first}
\theoremstyle{definition}
\usepackage{mathtools}
\numberwithin{equation}{section}

\newcommand{\ncom}{\newcommand}

\ncom{\beq}{\begin{equation}}
	\ncom{\eeq}{\end{equation}}
\ncom{\bea}{\begin{eqnarray*}}
	\ncom{\eea}{\end{eqnarray*}}
\ncom{\beqa}{\begin{eqnarray}}
	\ncom{\eeqa}{\end{eqnarray}}
\ncom{\nno}{\nonumber}
\ncom{\non}{\nonumber}
\ncom{\ds}{\displaystyle}
\ncom{\half}{\frac{1}{2}}
\ncom{\mbx}{\makebox{.25cm}}
\ncom{\hs}{\mbox{\hspace{.25cm}}}
\ncom{\rar}{\rightarrow}
\ncom{\Rar}{\Rightarrow}
\ncom{\noin}{\noindent}
\ncom{\bc}{\begin{center}}
	\ncom{\ec}{\end{center}}
\ncom{\sz}{\scriptsize}
\ncom{\rf}{\ref}
\ncom{\s}{\sqrt{2}}
\ncom{\sgm}{\sigma}
\ncom{\Sgm}{\Sigma}
\ncom{\psgm}{\sigma^{\prime}}
\ncom{\dt}{\delta}
\ncom{\Dt}{\Delta}
\ncom{\lmd}{\lambda}
\ncom{\Lmd}{\Lambda}
\ncom{\Th}{\Theta}
\ncom{\e}{\eta}
\ncom{\eps}{\epsilon}
\ncom{\pcc}{\stackrel{P}{>}}
\ncom{\lp}{\stackrel{L_{p}}{>}}
\ncom{\dist}{{\rm\,dist}}
\ncom{\sspan}{{\rm\,span}}
\ncom{\re}{{\rm Re\,}}
\ncom{\im}{{\rm Im\,}}
\ncom{\sgn}{{\rm sgn\,}}
\ncom{\ba}{\begin{array}}
	\ncom{\ea}{\end{array}}
\ncom{\hone}{\mbox{\hspace{1em}}}
\ncom{\htwo}{\mbox{\hspace{2em}}}
\ncom{\hthree}{\mbox{\hspace{3em}}}
\ncom{\hfour}{\mbox{\hspace{4em}}}
\ncom{\vone}{\vskip 2ex}
\ncom{\vtwo}{\vskip 4ex}
\ncom{\vonee}{\vskip 1.5ex}
\ncom{\vthree}{\vskip 6ex}
\ncom{\vfour}{\vspace*{8ex}}
\ncom{\norm}{\|\;\;\|}
\ncom{\integ}[4]{\int_{#1}^{#2}\,{#3}\,d{#4}}
\ncom{\vspan}[1]{{{\rm\,span}\{ #1 \}}}
\ncom{\dm}[1]{ {\displaystyle{#1} } }
\ncom{\ri}[1]{{#1} \index{#1}}

\newtheorem{theorem}{\bf Theorem}[section]
\newtheorem{remark}{\bf Remark}[section]
\newtheorem{proposition}{Proposition}[section]
\newtheorem{lemma}{Lemma}[section]

\newtheoremstyle
{remarkstyle}
{}
{11pt}
{}
{}
{\bfseries}
{:}
{     }
{\thmname{#1} \thmnumber{#2} }

\theoremstyle{remarkstyle}

\def\eps{\varepsilon}

\begin{document}
	\title{On Erlang Queue with Multiple Arrivals and its Time-changed Variant}
		\author[Rohini Bhagwanrao Pote]{Rohini Bhagwanrao Pote}
		\address{Rohini Bhagwanrao pote, Department of Mathematics, Indian Institute of Technology Bhilai, Durg 491002, India.}
		\email{rohinib@iitbhilai.ac.in}
	\author[Kuldeep Kumar Kataria]{Kuldeep Kumar Kataria}
	\address{Kuldeep Kumar Kataria, Department of Mathematics, Indian Institute of Technology Bhilai, Durg 491002, India.}
	\email{kuldeepk@iitbhilai.ac.in}
	\date{\today}
	\begin{abstract}
		We introduce and study a queue with the Erlang service system and whose arrivals are governed by a counting process in which there is a possibility of finitely many arrivals in an infinitesimal time interval. We call it the Erlang queue with multiple arrivals. Some of its distributional properties are obtained that includes the {\it state-phase} probabilities, the mean queue length and the distribution of busy period {\it etc.} Also, we study a time-changed variant of it by subordinating it with an independent inverse stable subordinator where we obtain its state probabilities and the mean queue length.
	\end{abstract}
	
	\maketitle 
	\section{Introduction}
	The Erlang queue is an important queueing theory model in which  arrivals occur according to the Poisson process and the service system has Erlang distribution. 	
	It has a wide range of applications in fields like telecommunications, finance, {\it etc}. For example, it is a useful model in call centers to predict number of agents required to handle incoming calls (see Spath and F{\"a}hnrich (2006)) and to model financial data (see Cahoy {\it et al.} (2015)). Luchak (1956), (1958) studied a single channel queue with Poisson arrivals and a general class of service-time distributions. Griffiths {\it et al.} (2006) obtained the transient solution of the Erlang queue.    
	
    Di Crescenzo {\it et al}. (2016) introduced and studied a generalization of  the Poisson process, namely, the generalized counting process (GCP). In GCP, there is a possibility of finitely many arrivals $1,2,\dots,l$  with positive rates $\lambda_{1}, \lambda_{2},\dots,\lambda_{l}$, respectively. It's a L\'evy process such that, in an infinitesimal interval of length $h$, the probability of no arrival is  $1-\sum_{j=1}^{l}\lambda_{j} h+o(h)$, an arrival of size $j\in\{1,2,\dots,l\}$ occurs with probability $\lambda_{j}h+o(h)$ and arrival of size more than $l$ has negligible probability, that is, of $o(h)$. For more details on GCP that include its superposition and thinning properties, and its time-changed variants, we refer the reader to  Kataria and Khandakar (2022), Dhillon and Kataria (2024a), Khandakar and Kataria (2024). Dhillon and Kataria (2024b) studied martingale characterizations of GCP and its time-changed variants. Recently, various growth processes with multiple transitions have been studied by many authors (see Chen {\it et al.} (2020), Vishwakarma and Kataria (2024), {\it etc}).
	
	A time-changed process is formed by composing a random process with a non-decreasing random time process. Generally, the involved component processes are assumed to be independent. Such processes are used in various scientific fields where time does not occur in regular or in deterministic manner. In the last decade, many time-changed processes have been studied, for example, the time fractional Poisson process (see Beghin and Orsingher (2009), Meerschaert {\it et al.} (2011)), the space fractional Poisson process (see Orsingher and Polito (2012)) and time-changed birth-death processes (see Orsingher and Polito (2011), Kataria and Vishwakarma (2025)), {\it etc}. 
	Cahoy {\it et al.} (2015) studied fractional $M/M/1$ queue where its state probabilities, algorithms to simulate $M/M/1$ queue and related birth-death process are obtained.
	Ascione {\it{et al.}} (2020) introduced and studied a time-changed Erlang queue where its state probabilities, mean queue length, distribution of busy period {\it etc}. are obtained. 
	
	In this paper, we introduce and study an Erlang queue where arrivals occur according to the GCP and its service system has the Erlang distribution. We call this queue as the Erlang queue with multiple arrivals.

	
	In Section \ref{sec3}, we define the Erlang queue with multiple arrivals and derive the system of difference-differential equations that governs its {\it state-phase} probabilities. We obtain the probability of no customers in the system at time $t\geq0$. Also, the explicit expression of transient {\it state-phase} probabilities are derived. Further, we define its queue length process, and obtain the mean queue length, the second moment of queue length process and distribution of its busy period.
	
	In Section \ref{sec4}, we study a time-changed variant of the Erlang queue with multiple arrivals. It is obtained by time-changing it with an independent inverse stable subordinator.
	The system of difference-differential equations that governs its state probabilities is derived. We obtain the state probabilities for time-changed Erlang queue with multiple arrivals. Also, the closed form expressions for its mean queue length is obtained.
	\section{Erlang queue with multiple arrivals}\label{sec3}
	In this section, we define a queue whose arrivals are modeled using GCP and it comprises of a service system which is composed of $k$ independent exponentially distributed phases each with mean $1/k \mu$. That is, its service system is distributed according to Erlang distribution with shape parameter $k$ and rate $\mu$.
	
	Let $\mathcal{N}(t)$ denote the number of customers in the system at time $t \geq 0$. For $\mathcal{N}(t)>0$, let $\mathcal{S}(t)$ denote the phase of a customer being served at time $t$ and for $\mathcal{N}(t)=0$, it is assumed that $\mathcal{S}(t)=0$. So, 
	\begin{equation*}
		\mathcal{Q}(t)=(\mathcal{N}(t),\mathcal{S}(t)),\, t\geq 0,
	\end{equation*}
	is a {\it state-phase} process with state space $\mathcal{H}_{0}=\mathcal{H} \cup \{(0,0)\}$, 
	where $\mathcal{H}=\{(n,s) \in \mathbb{N} \times \mathbb{N}: n \geq 1, \, 1 \leq s \leq k \} $ and $\mathbb{N}$ denotes the set of positive integers. We call this {\it state-phase} process $\{\mathcal{Q}(t)\}_{t\geq0}$ as the Erlang queue with multiple arrivals. It is assumed that there are no customers in the system at time 0, that is, $\mathcal{N}(0)=0$   and it is a single channel queue.
	
	For $t \geq 0$, let us denote the transient {\it state-phase} probabilities by
	\begin{equation*}
		p_{0}(t)=\mathrm{Pr}(\mathcal{Q}(t)=(0,0)|\mathcal{Q}(0)=(0,0))
	\end{equation*}
	and
	\begin{equation*}
		p_{n,s}(t)=\mathrm{Pr}(\mathcal{Q}(t)=(n,s)|\mathcal{Q}(0)=(0,0)), \, (n,s) \in \mathcal{H}, 
	\end{equation*}
	that is, $p_{0}(t)$ is the probability that there are no customers in the system at time $t$ and $p_{n,s}(t)$ is the probability that at time $t$ there are
	$n$ customers in the system and the customer
	in service is in phase $s$. As $\mathcal{S}(t)=0$ for $\mathcal{N}(t)=0$, we take $p_{0,s}(t)=0$, $1\leq s \leq k$.
	
	Further, it is assumed that the customers arrive according to GCP, that is, at any epoch, $m$  customers arrive with positive rate $\lambda_{m}$, $m=1,2, \dots , l$.  
	Let $\lambda=\lambda_{1}+\lambda_{2}+ \dots + \lambda_{l}$ and $c_{m}=\lambda_{m}/\lambda$ for all $m=1,2, \dots , l$. That is, in an infinitesimal time interval of length $h$, the arrival of size $m$ occurs with conditional probabilities  given by
		\begin{equation}
		q_{a}^{0}(m,h)=\mathrm{Pr}(\mathcal{Q}(t+h)=(m,k)|\mathcal{Q}(t)=(0,0))= \left\{
		\begin{array}{ll}
			1-\lambda h+o(h),\, m=0,\\
			\lambda c_{m}h+o(h),\,1\leq m\leq l,\\
			o(h),\, \text{otherwise},
		\end{array}
		\right. \label{1.2}
	\end{equation} 
	and
    \begin{equation}
    	q_{a}(m,h)=\mathrm{Pr}(\mathcal{Q}(t+h)=(i+m,s)|\mathcal{Q}(t)=(i,s))=\left\{
    	\begin{array}{ll}
    		1-\lambda h+o(h),\, m=0,\\
    		\lambda c_{m}h+o(h),\,1\leq m\leq l,\\
    		o(h),\, \text{otherwise},
    	\end{array}
    	\right.   \label{1.1}
    \end{equation}
	where $i\in \mathbb{N}$ and $1\leq s\leq k$. 
	
	Also, in such a queue, the
	service system is composed of $k$ phases each distributed exponentially with mean $1/k\mu$. Thus, in an infinitesimal time interval of length $h$, the departure of size $r$ occurs with conditional probabilities given by
	\begin{align}
		q_{d}^{0}(0,h)&=\mathrm{Pr}(\mathcal{Q}(t+h)=(0,s)|\mathcal{Q}(t)=(0,0))=
		\begin{cases}
			1,\,s=0,\\
			0,\, \text{otherwise},
		\end{cases}    \label{1.4}
	\end{align}
	and
	\begin{align}
		q_{d}(r,h)&=\mathrm{Pr}(\mathcal{Q}(t+h)=(i,s-r)|\mathcal{Q}(t)=(i,s))=
		\begin{cases}
			1-k\mu h+o(h),\, r=0,\\
			k\mu h+o(h),\, r=1,\\
			o(h),\, \text{otherwise},
		\end{cases}    \label{1.3}
	\end{align}
	where $i \in \mathbb{N}$ and $1\leq s\leq k$.
	
	Now, we obtain the system of differential equations that governs the {\it state-phase} probabilities of the Erlang queue with multiple arrivals. 
	
	\begin{theorem}\label{t3.1}
		The {\it state-phase} probabilities of $\{\mathcal{Q}(t)\}_{t\geq 0}$ satisfy the following system of difference-differential equations:
			\begin{align*}
				\frac{\mathrm {d}}{\mathrm {d}t}p_{0}(t)
				&=-\lambda p_{0}(t)+k\mu p_{1,1}(t), \nonumber \\
				\frac{\mathrm {d}}{\mathrm {d}t}p_{1,s}(t)
				&=-(\lambda + k \mu)p_{1,s}(t)+k\mu p_{1,s+1}(t), \, 1 \leq s \leq k-1, \nonumber \\
				\frac{\mathrm {d}}{\mathrm {d}t}p_{1,k}(t)
				&=-(\lambda + k \mu)p_{1,k}(t)+k\mu p_{2,1}(t)+ \lambda c_{1}p_{0}(t),  \\
				\frac{\mathrm {d}}{\mathrm {d}t}p_{n,s}(t)
				&=-(\lambda + k \mu)p_{n,s}(t)+k\mu p_{n,s+1}(t)+\lambda \sum_{m=1}^{\text{min}\{n,l\}}c_{m}p_{n-m,s}(t), \,  1\leq s \leq k-1,  \, n \geq 2,  \\
				\frac{\mathrm {d}}{\mathrm {d}t}p_{n,k}(t)
				&=-(\lambda + k \mu)p_{n,k}(t)+k\mu p_{n+1,1}(t)+\lambda \sum_{m=1}^{n-1}c_{m}p_{n-m,k}(t)+\lambda c_{n}p_{0}(t),\, 2\leq n\leq l, \nonumber \\
				\frac{\mathrm {d}}{\mathrm {d}t}p_{n,k}(t)&=-(\lambda + k \mu)p_{n,k}(t)+k\mu p_{n+1,1}(t)+ \lambda \sum_{m=1}^{l}c_{m}p_{n-m,k}(t),\,n >l,
		\end{align*}
		with initial conditions 
		$p_{0}(0)=1$ and
		$p_{n,s}(0)=0$, $1\leq s\leq k$, $n\geq 1$.
	\end{theorem}
	\begin{proof}
		Let us consider an infinitesimal time interval of length $h$ such that $o(h)/h\to 0 $ as $h\to0$. Then, we have the following possibilities: 
		\begin{align*}
			p_{0}(t+h)
			&=p_{0}(t)q_{a}^{0}(0,h)q_{d}^{0}(0,h)+p_{1,1}(t)q_{a}(0,h)q_{d}(1,h),\\
			p_{1,s}(t+h)
			&=p_{1,s}(t)q_{a}(0,h)q_{d}(0,h)+p_{1,s+1}(t)q_{a}(0,h)q_{d}(1,h),\, 1\leq s \leq k-1,\\p_{1,k}(t+h)
			&=p_{1,k}(t)q_{a}(0,h)q_{d}(0,h)+p_{2,1}(t)q_{a}(0,h)q_{d}(1,h)+ p_{0}(t)q_{a}^{0}(1,h)q_{d}^{0}(0,h), \\
			p_{n,s}(t+h)
			&=p_{n,s}(t)q_{a}(0,h)q_{d}(0,h)+p_{n,s+1}(t)q_{a}(0,h)q_{d}(1,h)\\
			&+\sum_{m=1}^{\text{min} \{n,l\}}p_{n-m,s}(t)q_{a}(m,h)q_{d}(0,h),\, n \geq 2,\, 1 \leq s \leq k-1,\\
			p_{n,k}(t+h)=
			&p_{n,k}(t)q_{a}(0,h)q_{d}(0,h)+p_{n+1,1}(t)q_{a}(0,h)q_{d}(1,h)\\
			&+\sum_{m=1}^{n-1}p_{n-m,k}(t)q_{a}(m,h)q_{d}(0,h)+p_{0}(t)q_{a}^{0}(n,h)q_{d}^{0}(0,h), \,2\leq n\leq l
		\end{align*}
		and
		\begin{align*}
			p_{n,k}(t+h)=
			&p_{n,k}(t)q_{a}(0,h)q_{d}(0,h)+p_{n+1,1}(t)q_{a}(0,h)q_{d}(1,h)\\
			&+\sum_{m=1}^{l}p_{n-m,k}(t)q_{a}(m,h)q_{d}(0,h), \,n>l.
		\end{align*}
		On using (\ref{1.1})-(\ref{1.4}), we get 
		\begin{align*}
			\frac{p_{0}(t+h)-p_{0}(t)}{h}=
			&-\lambda p_{0}(t)+ k\mu p_{1,1}(t)+h(-\lambda k\mu p_{1,1}(t))+\frac{o(h)}{h}, \\
			\frac{p_{1,s}(t+h)-p_{1,s}(t)}{h}=
			&-(\lambda + k \mu ) p_{1,s}(t)+ k\mu p_{1,s+1}(t)+h(\lambda k\mu p_{1,s}(t)\\
			&-\lambda k\mu p_{1,s+1}(t))+\frac{o(h)}{h}, \\
			\frac{p_{1,k}(t+h)-p_{1,k}(t)}{h}=
			&-(\lambda + k \mu ) p_{1,k}(t)+ k\mu p_{2,1}(t)+\lambda c_{1}p_{0}(t)+h(\lambda k\mu p_{1,k}(t)-\lambda k\mu p_{2,1}(t))\\
			&+\frac{o(h)}{h}, \\
			\frac{p_{n,s}(t+h)-p_{n,s}(t)}{h}=
			&-(\lambda + k \mu ) p_{n,s}(t)+ k\mu p_{n,s+1}(t)+\lambda \sum_{m=1}^{\text{min}\{n,l\}} c_{m}p_{n-m,s}(t)+h\big(\lambda k\mu p_{n,s}(t)\\
			&-\lambda k\mu p_{n,s+1}(t)-\lambda k\mu \sum_{m=1}^{\text{min}\{n,l\}}c_{m}p_{n-m,s}(t)\big)+\frac{o(h)}{h},\\ \frac{p_{n,k}(t+h)-p_{n,k}(t)}{h}=
			&-(\lambda + k \mu) p_{n,k}(t)+ k\mu p_{n+1,1}(t)+\lambda\Big( \sum_{m=1}^{n-1}c_{m}p_{n-m,k}(t)+c_{n}p_{0}(t)\Big)\\
			&+h\big(\lambda k\mu p_{n,k}(t)-\lambda k\mu p_{n+1,1}(t)-\lambda k\mu \sum_{m=1}^{n-1}c_{m}p_{n-m,k}(t)\big)+\frac{o(h)}{h}
		\end{align*}
		and
		\begin{align*}
			\frac{p_{n,k}(t+h)-p_{n,k}(t)}{h}&=-(\lambda + k \mu) p_{n,k}(t)+ k\mu p_{n+1,1}(t)+\lambda \sum_{m=1}^{l}c_{m}p_{n-m,k}(t)\\
			&\ \ +h\big(\lambda k\mu p_{n,k}(t)-\lambda k\mu p_{n+1,1}(t)-\lambda k\mu \sum_{m=1}^{l}c_{m}p_{n-m,k}(t)\big)+\frac{o(h)}{h}.
		\end{align*}		
		On letting $h\to0$, we get the required result.
	\end{proof}
	\begin{remark}
		For $l=1$, the result in Theorem \ref{t3.1} reduces to the corresponding result for $M/E_{k}/1$ queue (see Griffiths {\it et al.} (2006), Eq. (4)).
	\end{remark}			
	Let us consider the following generating function:
	\begin{align}\label{3.2a}
		G(x,t)&=p_{0}(t)+\sum_{n=1}^{\infty} \sum_{s=1}^{k}x^{(n-1)k+s}p_{n,s}(t), \,t\geq0.
	\end{align}
	Note that $G(x,t)$ is analytic in $|x|\leq1$.
	
	\begin{proposition}\label{prep1}
		The generating function of $\{\mathcal{Q}(t)\}_{t\geq0}$ solves \begin{equation}\label{3.2}
			x \frac{\partial}{\partial t}G(x,t)=\Big(k\mu(1-x)-\lambda x\Big(1-\sum_{m=1}^{l}c_{m}x^{mk}\Big)\Big)G(x,t)-k\mu(1-x)p_{0}(t),
		\end{equation}
		with initial condition $G(x,0)=1$.
	\end{proposition}
	\begin{proof}
		From (\ref{3.2a}) and using Theorem \ref{t3.1}, we get {\footnotesize\begin{align*}
				x\frac{\partial}{\partial t}G(x,t)&=x\frac{\mathrm{d}}{\mathrm{d}t}p_{0}(t)+\sum_{s=1}^{k-1}x^{s+1}\frac{\mathrm{d}}{\mathrm{d}t}p_{1,s}(t)+x^{k+1}\frac{\mathrm{d}}{\mathrm{d}t}p_{1,k}(t)\\
				&\ \ +\sum_{n=2}^{\infty} \sum_{s=1}^{k-1}x^{k(n-1)+s+1} \frac{\mathrm{d}}{\mathrm{d}t}p_{n,s}(t) +\sum_{n=2}^{l}x^{nk+1}\frac{\mathrm{d}}{\mathrm{d}t}p_{n,k}(t)+\sum_{n=l+1}^{\infty}x^{nk+1}\frac{\mathrm{d}}{\mathrm{d}t}p_{n,k}(t)\\
				&=x(-\lambda p_{0}(t)+k\mu p_{1,1}(t))+\sum_{s=1}^{k-1}x^{s+1}(-(\lambda +k \mu)p_{1,s}(t)+k\mu p_{1,s+1}(t))\\
				&\ \ +x^{k+1}(-(\lambda +k \mu)p_{1,k}(t)+k\mu p_{2,1}(t)+\lambda c_{1}p_{0}(t))\\
				&\ \ +\sum_{n=2}^{\infty} \sum_{s=1}^{k-1}x^{k(n-1)+s+1}\Big(-(\lambda +k \mu)p_{n,s}(t)+k\mu p_{n,s+1}(t)+\lambda \sum_{m=1}^{\text{min}\{n,l\}}c_{m} p_{n-m,s}(t)\Big)\\
				&\ \ +\sum_{n=2}^{l} x^{nk+1}\Big(-(\lambda +k \mu)p_{n,k}(t)+k\mu p_{n+1,1}(t) +\lambda\sum_{m=1}^{n-1}c_{m}p_{n-m,k}(t)+\lambda  c_{n}p_{0}(t)\Big) \\
				&\ \ +\sum_{n=l+1}^{\infty} x^{nk+1}\Big(-(\lambda +k \mu)p_{n,k}(t)+k\mu p_{n+1,1}(t)+\lambda\sum_{m=1}^{l}c_{m}p_{n-m,k}(t)\Big)\\
				&=-(\lambda+k\mu)xG(x,t)+k\mu xp_{0}(t)+k\mu (G(x,t)-p_{0}(t))\\
				&\ \ +\lambda\Big(\sum_{m=1}^{l}c_{m}x^{mk+1}p_{0}(t)+\sum_{n=1}^{\infty}\sum_{s=1}^{k}x^{nk+s+1}\sum_{m=1}^{\text{min}\{n+1,l\}}c_{m}p_{n+1-m,s}(t)\Big) \\
				&=-(\lambda+k\mu)xG(x,t)+k\mu xp_{0}(t)+k\mu (G(x,t)-p_{0}(t))+\lambda\Big(\sum_{m=1}^{l}c_{m}x^{mk+1}p_{0}(t)\\
				&\ \ +c_{1}\sum_{n=1}^{\infty}\sum_{s=1}^{k}x^{nk+s+1}p_{n,s}(t)+c_{2}\sum_{n=2}^{\infty}\sum_{s=1}^{k}x^{nk+s+1}p_{n-1,s}(t)+\dots +c_{l}\sum_{n=l}^{\infty}\sum_{s=1}^{k}x^{nk+s+1}p_{n+1-l,s}(t)\Big)\\
				&=-(\lambda+k\mu)xG(x,t)+k\mu xp_{0}(t)+k\mu (G(x,t)-p_{0}(t))\\
				&\  \ +\lambda\Big(\sum_{m=1}^{l}c_{m}x^{mk+1}p_{0}(t)+ \Big(\sum_{m=1}^{l}c_{m}x^{mk+1}\Big)(G(x,t)-p_{0}(t))\Big)\\
				&=\Big(k\mu(1-x)-\lambda x\Big(1-\sum_{m=1}^{l}c_{m}x^{mk}\Big)\Big)G(x,t)-k\mu(1-x)p_{0}(t)
		\end{align*}}
		with $G(x,0)=1$. This completes the proof.
	\end{proof}
	\begin{remark}
	For $l=1$, the Preposition \ref{prep1} reduces to the corresponding result for $M/E_{k}/1$ queue (see Griffiths {\it et al.} (2006), Eq. (6)).
	\end{remark}
	Let $\tilde{f}$ denote the Laplace transform of the function $f$. On taking the Laplace transform on both sides of (\ref{3.2}), we get
	\begin{equation}
		x(z\tilde{G}(x,z)-1)=\Big(k\mu(1-x)-\lambda x\Big(1-\sum_{m=1}^{l}c_{m}x^{mk}\Big)\Big)\tilde{G}(x,z)-k\mu(1-x)\tilde{p_{0}}(z).\nonumber\\ 
	\end{equation}
	Thus,
	\begin{equation*}
		\tilde{G}(x,z)=\frac{x-k\mu(1-x)\tilde{p_{0}}(z)}{(\lambda + k \mu +z)x-k \mu - \lambda x\sum_{m=1}^{l}c_{m}x^{mk}}, \, z>0.  
	\end{equation*}
	
	Let $f(x)=(\lambda+k\mu+z)x$ and $g(x)=-k\mu-\lambda x\sum_{m=1}^{l}c_{m}x^{mk}$. Note that $|g(x)|\leq |f(x)|$ on $|x|=1$. By Rouche's theorem, $p(x)=f(x)+g(x)$ has a single zero in $|x|\leq1$, say at $\theta$. As $\tilde{G}(x,z)$ converges for $|x| \leq 1$, we have  $\theta-k\mu(1-\theta)\tilde{p_{0}}(z)=0$. So,
	
	\begin{equation}\label{3.5}
		\tilde{p_{0}}(z)=\frac{\theta}{k\mu(1-\theta)}.
	\end{equation}
	
	Let $\delta = k\mu/(\lambda + k\mu +z)$, $\omega=\lambda/(\lambda + k\mu + z)$ and $\phi(x)=x\sum_{m=1}^{l}c_{m}x^{mk}$. The Lagrange's expansion of an arbitrary holomorphic function $F(\theta)$ can be written as 
	\begin{equation}
		F(\theta)=F(\delta)+\sum_{n=1}^{\infty} \frac{\omega^{n}}{n!}\frac{\mathrm{d}^{n-1}}{\mathrm{d} \delta^{n-1}}(F'(\delta) \{\phi(\delta)\}^{n}).  \label{3.9}
	\end{equation}	
	Let 
	\begin{equation}\label{3.10}
		F(\theta)=\theta^{\sigma},\, \sigma \in \mathbb{N}.
	\end{equation}		
	So,			
	\begin{align}\label{derrivative}
		F'(\delta) \{\phi(\delta)\}^{n}&=\sigma\delta^{\sigma+n-1}(c_{1}\delta^{k}+c_{2}\delta^{2k}+\dots+c_{l}\delta^{lk})^{n}\nonumber \\
		&=\sigma\sum_{\substack{\sum_{j=1}^{l}m_{j}=n \\ m_{j}\in \mathbb{N}_{0}}}{n!}{\Big(\prod_{i=1}^{l}\frac{c_{i}^{m_{i}}}{m_{i}!}}\Big)\delta^{\sigma+n-1+k\sum_{j=1}^{l}jm_{j}}, 
	\end{align}
	where $\mathbb{N}_{0}=\mathbb{N}\cup\{0\}$.	On taking $(n-1)$th derivative in (\ref{derrivative}), we get			
	\begin{equation}\label{3.12}
		\frac{\mathrm{d}^{n-1}}{\mathrm{d} \delta^{n-1}}(F'(\delta) \{\phi(\delta)\}^{n})=
		\sigma\sum_{\substack{\sum_{j=1}^{l}m_{j}=n \\ m_{j}\in \mathbb{N}_{0}}}{n!}{\Big(\prod_{i=1}^{l}\frac{c_{i}^{m_{i}}}{m_{i}!}}\Big)\frac{(\sigma+n-1+k\sum_{j=1}^{l}jm_{j})!}{(\sigma+k\sum_{j=1}^{l}jm_{j})!}\delta^{\sigma+k\sum_{j=1}^{l}jm_{j}}.  
	\end{equation}		
	By substituting (\ref{3.10}) and (\ref{3.12}) in (\ref{3.9}), we obtain
	\begin{equation}\label{3.13}
		\theta^{\sigma}=
		\delta^{\sigma} +\sigma\sum_{n=1}^{\infty}\Big(\frac{\lambda}{k\mu}\Big)^{n}\sum_{\substack{\sum_{j=1}^{l}m_{j}=n \\ m_{j}\in \mathbb{N}_{0}}}\Big({\prod_{i=1}^{l}\frac{c_{i}^{m_{i}}}{m_{i}!}}\Big)\frac{(\sigma+n-1+k\sum_{j=1}^{l}jm_{j})!}{(\sigma+k\sum_{j=1}^{l}jm_{j})!}\delta^{\sigma+n+k\sum_{j=1}^{l}jm_{j}}. 
	\end{equation}		
	Let $\eta=\lambda/k\mu$ and $\tau=k\mu t$. On taking the inverse Laplace transform in (\ref{3.13}), we get
	\begin{equation} \label{3.14}
		\mathbb{L}^{-1}(\theta^{\sigma})=\frac{k\mu\sigma}{\tau}\Big(\frac{\tau^{\sigma}}{\sigma!} +\sum_{n=1}^{\infty}\eta^{n}\sum_{\substack{\sum_{j=1}^{l}m_{j}=n \\ m_{j}\in \mathbb{N}_{0}}}\Big({\prod_{i=1}^{l}\frac{c_{i}^{m_{i}}}{m_{i}!}}\Big)\frac{\tau^{\sigma+n+k\sum_{j=1}^{l}jm_{j}}}{(\sigma+k\sum_{j=1}^{l}jm_{j})!}\Big)e^{-(1+\eta)\tau}. 
	\end{equation}
	From (\ref{3.5}), we have
	\begin{equation}\label{lpp0}
		\tilde{p_{0}}(z)=\frac{1}{k\mu}\sum_{h=1}^{\infty}\theta^{h}.
	\end{equation}	
	Finally, on taking the inverse Laplace transform in (\ref{lpp0}) and using (\ref{3.14}), we obtain the zero {\it state-phase} probability of $\{Q(t)\}_{t\geq0}$ in the following form: 
	\begin{equation}\label{3.15}
		p_{0}(t)=\frac{1}{\tau}\sum_{h=1}^{\infty}h\sum_{n=0}^{\infty}\eta^{n}\sum_{\substack{\sum_{j=1}^{l}m_{j}=n \\ m_{j}\in \mathbb{N}_{0}}}{\Big(\prod_{i=1}^{l}\frac{c_{i}^{m_{i}}}{m_{i}!}}\Big)\frac{\tau^{h+n+k\sum_{j=1}^{l}jm_{j}}}{(h+k\sum_{j=1}^{l}jm_{j})!}e^{-(1+\eta)\tau},\,t\geq0.
	\end{equation}
	
	
	\begin{remark}
		For $l=1$, the result in (\ref{3.15}) reduces to that for $M/E_{k}/1$ queue (see Griffiths {\it et al.} (2006), Eq. (18)).
	\end{remark}
	\begin{lemma}\label{lemma}
		Let $f$ be an arbitrary function. Then, the following identity holds:
		\begin{equation}\label{identt}
			\sum_{\substack{m_{j}\geq0\\1\leq j\leq l}}\sum_{r=0}^{\infty}f(m_{1},m_{2}\dots,m_{l},r)=\sum_{\substack{m_{j}\geq0\\1\leq j\leq l}}\sum_{r=0}^{\infty}\sum_{w=1}^{k}f(m_{1},m_{2},\dots,m_{l},rk+w-1),
		\end{equation}
		provided the series in \eqref{identt} converge absolutely.
	\end{lemma}
	\begin{proof}
		From right hand side of \eqref{identt}, we have
		{\footnotesize	\begin{align*}
				\sum_{\substack{m_{j}\geq0\\1\leq j\leq l}}\sum_{r=0}^{\infty}\sum_{w=1}^{k}f(m_{1},m_{2},\dots,m_{l},rk+w-1)&=\sum_{\substack{m_{j}\geq0\\1\leq j\leq l}}\sum_{r=0}^{\infty}\Big(f(m_{1},m_{2},\dots,m_{l},rk)+f(m_{1},m_{2},\dots,m_{l},rk+1)\\
				&\ \ +\dots+f(m_{1},m_{2},\dots,m_{l},(r+1)k-1)\Big)\\
				&=\sum_{\substack{m_{j}\geq0\\1\leq j\leq l}}\Big(f(m_{1},\dots,m_{l},0)+f(m_{1},\dots,m_{l},1)\\
				&\ \ +\dots+f(m_{1},\dots,m_{l},k-1)+f(m_{1},\dots,m_{l},k)\\
				&\ \ +f(m_{1},\dots,m_{l},k+1)+\dots+f(m_{1},\dots,m_{l},2k-1)+\dots\Big)\\
				&=\sum_{\substack{m_{j}\geq0\\1\leq j\leq l}}\sum_{r=0}^{\infty}f(m_{1},m_{2},\dots,m_{l},r),
		\end{align*}}
		which is equal to the left hand side of \eqref{identt}.
	\end{proof}
	\begin{remark}
		For $l=1$, Lemma \ref{lemma} reduces to the result given in Eq. (25) of Griffiths {\it et al.} (2006).
	\end{remark}
	\begin{theorem}\label{t3.2}
		For $t\geq0$ and $1\leq s\leq k-1$, the transient {\it state-phase} probabilities are given by
		\begin{align}\label{pnst}
			p_{n,s}(t)
			&=e^{-(\lambda+k\mu)t}\sum_{\substack{m_{j},r\geq0\\\sum_{j=1}^{l}jm_{j} -r=n}}\Big(\prod_{i=1}^{l}\frac{c_{i}^{m_{i}}}{m_{i}!}\Big)\Bigg(\frac{\lambda^{\sum_{j=1}^{l}m_{j}}(k\mu )^{k(r+1)-s}}{(k(r+1)-s)!} \Big(t^{\sum_{j=1}^{l}m_{j} +k(r+1)-s} \nonumber \\
			& \hspace{1cm} +k\mu\int_{0}^{t}p_{0}(z)e^{(\lambda+k\mu)z}(t-z)^{\sum_{j=1}^{l}m_{j}+k(r+1)-s}\mathrm{d}z\Big) -\frac{\lambda^{\sum_{j=1}^{l}m_{j}}(k\mu )^{k(r+1)-s}}{(k(r+1)-s-1)!}\nonumber\\
			& \hspace{4cm} \cdot\int_{0}^{t}p_{0}(z)e^{(\lambda+k\mu)z}(t-z)^{\sum_{j=1}^{l}m_{j}+k(r+1)-s-1}\mathrm{d}z\Bigg), n\geq1
		\end{align}
		and for $s=k$, these are given by
	\begin{align}
			p_{n,k}(t)	&=e^{-(\lambda+k\mu)t}\sum_{\substack{m_{j},r\geq0\\\sum_{j=1}^{l}jm_{j}-r=n}}\Bigg(\Big(\prod_{i=1}^{l}\frac{c_{i}^{m_{i}}}{m_{i}!}\Big)\frac{\lambda^{\sum_{j=1}^{l}m_{j}}(k\mu )^{rk}}{(rk)!} \Big(t^{\sum_{j=1}^{l}m_{j}+rk}\nonumber\\
			&\ \ +k\mu\int_{0}^{t}p_{0}(z)e^{(\lambda+k\mu)z}(t-z)^{\sum_{j=1}^{l}m_{j}+rk}\mathrm{d}z\Big)-\Big(\frac{c_{1}^{m_{1}+1}}{(m_{1}+1)!}\Big)\Big(\prod_{i=2}^{l}\frac{c_{i}^{m_{i}}}{m_{i}!}\Big)\nonumber\\
			&\hspace{0.5cm} \cdot\frac{\lambda^{\sum_{j=1}^{l}m_{j}+1}(k\mu )^{(r+1)k}}{((r+1)k-1)!} \int_{0}^{t}p_{0}(z)e^{(\lambda+k\mu)z}(t-z)^{\sum_{j=1}^{l}m_{j}+k(r+1)}\mathrm{d}z\Bigg),\, n\geq1.\label{pnkt}
		\end{align}
	\end{theorem}
	\begin{proof}
		From (\ref{3.2}), we have 
		\begin{equation}\label{3.16}
			\frac{\partial}{\partial t}G(x,t)=H(x)G(x,t)-k\mu \left( \frac{1-x}{x} \right)p_{0}(t),\, G(x,0)=1,   			 	 
		\end{equation}
		where
		\begin{equation}\label{3.17}
			H(x)
			=\lambda\sum_{m=1}^{l}c_{m}x^{mk}+\frac{k\mu}{x}-(\lambda+k\mu).  
		\end{equation}
		Note that the solution of (\ref{3.16}) can be written in the following form:
		\begin{equation}\label{3.18}
			G(x,t)=A(t)e^{H(x)t},  
		\end{equation}
		with $A(0)=1$. Thus, we have
		\begin{equation}\label{A't}
			A'(t)=-k \mu \Big(\frac{1-x}{x}\Big)p_{0}(t)e^{-H(x)t}.
		\end{equation}
		On integrating (\ref{A't}) from 0 to $t$, we get
		\begin{equation}\label{A(t)}
			A(t)=A(0)-k \mu \Big(\frac{1-x}{x}\Big) \int_{0}^{t}p_{0}(z)e^{-H(x)z}\mathrm{d}z.
		\end{equation} 
		By substituting (\ref{3.17}) and (\ref{A(t)}) in (\ref{3.18}), we obtain
		\begin{align}
			G(x,t)
			&=e^{-(\lambda+k \mu)t}\Big(e^{\big(\lambda\sum_{m=1}^{l}c_{m}x^{mk}+\frac{k\mu}{x}\big)t}-k\mu \Big(\frac{1-x}{x}\Big) \nonumber \\ &\hspace{5.4cm}\,\cdot\int_{0}^{t}p_{0}(z)e^{(\lambda + k \mu)z}e^{\big(\lambda \sum_{m=1}^{l}c_{m}x^{mk}+\frac{k\mu}{x}\big)(t-z)}\mathrm{d}z\Big).  \label{3.19}
		\end{align}
		
		For $y\geq0$, we have
		\begin{align}
			e^{(\lambda\sum_{m=1}^{l}c_{m}x^{mk}+\frac{k\mu}{x})y}\nonumber
			&=e^{\frac{yk\mu}{x}}\prod_{i=1}^{l}e^{\lambda y c_{i}x^{ik}}\nonumber\\
			&=\sum_{\substack{m_{j}\geq0\\1\leq j\leq l}}\sum_{r=0}^{\infty}\Big(\prod_{i=1}^{l}\frac{(\lambda y c_{i}x^{ik})^{m_{i}}}{m_{i}!}\Big){\frac{(yk\mu/x)}{r!}}^{r}\nonumber\\
			&=\sum_{\substack{m_{j}\geq0\\1\leq j\leq l}}\sum_{r=0}^{\infty}\sum_{w=1}^{k}\Big(\prod_{i=1}^{l}\frac{c_{i}^{m_{i}}}{m_{i}!}\Big)\frac{\lambda^{\sum_{j=1}^{l}m_{j}}(k\mu )^{rk+w-1}}{(rk+w-1)!}\nonumber\\
			&\hspace{2cm}\cdot y^{\sum_{j=1}^{l}m_{j}+rk+w-1}x^{k(\sum_{j=1}^{l}jm_{j}-r)-w+1}\nonumber,\,\,(\text{by using Lemma \ref{lemma}})\\
			&=\sum_{\substack{m_{j}\geq0\\1\leq j\leq l}}\sum_{r=0}^{\infty}\sum_{s=1}^{k}\Big(\prod_{i=1}^{l}\frac{c_{i}^{m_{i}}}{m_{i}!}\Big)\frac{\lambda^{\sum_{j=1}^{l}m_{j}}(k\mu )^{k(r+1)-s}}{(k(r+1)-s)!}
			\nonumber\\
			&\hspace{4.5cm}\cdot y^{\sum_{j=1}^{l}m_{j}+k(r+1)-s}x^{k(\sum_{j=1}^{l}jm_{j}-r-1)+s},\label{exp}
		\end{align}
		where the last step follows from the change of variable $s=k-w+1$.
		
		By using (\ref{exp}) in (\ref{3.19}), we have
		{\footnotesize \begin{align}\label{gxt}
				G(x,t)
				&=
				e^{-(\lambda+k\mu)t}\sum_{\substack{m_{j}\geq0\\1\leq j\leq l}}\sum_{r=0}^{\infty}\sum_{s=1}^{k}\Big(\prod_{i=1}^{l}\frac{c_{i}^{m_{i}}}{m_{i}!}\Big)\frac{\lambda^{\sum_{j=1}^{l}m_{j}}(k\mu )^{k(r+1)-s}}{(k(r+1)-s)!}\bigg(\Big(t^{\sum_{j=1}^{l}m_{j}+k(r+1)-s}\nonumber\\
				&\hspace{1.5cm} +k\mu\int_{0}^{t}p_{0}(z)e^{(\lambda+k\mu)z}(t-z)^{\sum_{j=1}^{l}m_{j}+k(r+1)-s}\mathrm{d}z\Big)x^{k(\sum_{j=1}^{l}jm_{j}-r-1)+s}\nonumber\\
				&\hspace{2cm}  -k\mu\int_{0}^{t}p_{0}(z)e^{(\lambda+k\mu)z}(t-z)^{\sum_{j=1}^{l}m_{j}+k(r+1)-s}\mathrm{d}z x^{k(\sum_{j=1}^{l}jm_{j}-r-1)+s-1}\bigg).
		\end{align}}
		
		Finally, the required result follows by comparing the coefficients of $x^{k(\sum_{j=1}^{l}jm_{j}-r-1)+s}$ in \eqref{3.2a} and (\ref{gxt}).
	\end{proof}
	\begin{remark}
		For $l=1$, the transient {\it state-phase} probabilities \eqref{pnst} and \eqref{pnkt} reduces to that of $M/E_{k}/1$ queue (see Ascione {\it{et al.}} (2020), Eq. (6) and Eq. (7)).
	\end{remark}
	\subsection{Queue length process}\label{subsec3.1}
	    Here, we give an equivalent representation of the Erlang queue with multiple arrivals in terms of a queue length process. 
	    
	    Let us consider the following  bijective map $g_{k}:\mathcal{H}_{0}\rightarrow\mathbb{N}_{0}$ as described in Ascione {\it{et al.}} (2020), p. 3252:  
	    \begin{equation}\label{gk}
	    	g_{k}(n,s)=\left\{
	    	\begin{array}{ll}
	    		k(n-1)+s,\, (n,s)\in\mathcal{H},\\
	    		0,\, (n,s)=(0,0).
	    	\end{array}
	    	\right.  
	    \end{equation}
	   Its the inverse map is given by $(a_{k}(m),b_{k}(m))$, where 
	   \begin{equation*}
	   	b_{k}(m)=\left\{
	   	\begin{array}{ll}
	   		\min\{s>0: s \equiv m \pmod{k} \},\, m>0,\\
	   		0,\, m=0,
	   	\end{array}
	   	\right.  
	   \end{equation*}
	    and
	   \begin{equation*}
	   	a_{k}(m)=\left\{
	   	\begin{array}{ll}
	   		\frac{m-b_{k}(m)}{k}+1,\, m>0,\\
	   		0,\, m=0.
	   	\end{array}
	   	\right.  
	   \end{equation*}
	    Let us consider the following queue length process: $\mathcal{L}(t)=g_{k}(\mathcal{Q}(t))$, $t\geq0$. 
	    The {\it state-phase} process $\{\mathcal{Q}(t)\}_{t\geq0}$, that is, the Erlang queue with multiple arrivals can be equivalently represented by the queue length process $\{\mathcal{L}(t)\}_{t\geq0}$ in terms of phases as $g_{k}(\cdot)$ is a bijection. So, each state of $\{\mathcal{L}(t)\}_{t\geq0}$   represents a unique state of $\{\mathcal{Q}(t)\}_{t\geq0}$ and vice-versa. 
	    
	    Let $p_{n}(t)=\mathrm{Pr}(\mathcal{L}(t)=n|\mathcal{L}(0)=0)$, $n\geq0$ denote its state probabilities. Then, in an infinitesimal time interval of length $h$ such that $o(h)/h\rightarrow0$ as
	    $h\rightarrow0$, we have 
	    {\small\begin{align*}
	    	p_{0}(t+h)&=p_{0}(t)q^{0}_{a}(0,h)q_{d}^{0}(0,h)+p_{1}(t)q_{a}(0,h)q_{d}(1,h),\\
	    	p_{n}(t+h)&=p_{n}(t)q_{a}(0,h)q_{d}(0,h)+p_{n+1}(t)q_{a}(0,h)q_{d}(1,h)+\sum_{m=1}^{l}p_{n-mk}(t)q_{a}(m,h)q_{d}(0,h), \, n\geq1.
	    \end{align*}}
	    On using (\ref{1.1})-(\ref{1.4}), we get
	    \begin{align*}
	    	\frac{p_{0}(t+h)-p_{0}(t)}{h}=
	    	&-\lambda p_{0}(t)+ k\mu p_{1}(t)+h(-\lambda k\mu p_{1}(t))+\frac{o(h)}{h},\\
	    	\frac{p_{n}(t+h)-p_{n}(t)}{h}=
	    	&-(\lambda+k\mu) p_{n}(t)+ k\mu p_{n+1}(t)+\lambda\sum_{m=1}^{l}c_{m}p_{n-mk}(t)\\
	    	&\ \ +h\Big(\lambda k\mu p_{n}(t)-\lambda k\mu p_{n+1}(t)-\lambda k\mu\sum_{m=1}^{l}c_{m}p_{n-mk}(t)\Big)+\frac{o(h)}{h},\,n\geq1.
	    \end{align*}		
	    On letting $h\rightarrow0$, we get the following system of differential equations that governs the state probabilities of the queue length process $\{\mathcal{L}(t)\}_{t\geq0}$:
	    \begin{align}\label{lengthDE}
	    	\left.
	    	\begin{aligned}
	    		\frac{\mathrm {d}}{\mathrm {d}t}p_{0}(t)
	    		&=-\lambda p_{0}(t)+k\mu p_{1}(t),  \\
	    		\frac{\mathrm {d}}{\mathrm {d}t}p_{n}(t)
	    		&=-(\lambda + k \mu)p_{n}(t)+k\mu p_{n+1}(t)+\lambda\sum_{m=1}^{l}c_{m}p_{n-mk}(t), \, n \geq 1,
	    	\end{aligned}
	    	\right\}
	    \end{align}
	    with $p_{0}(0)=1$.
	    
	     Note that 
	     \begin{equation}\label{gkn}
	     	p_{n}(t)=p_{a_{k}(n),b_{k}(n)}(t),\,n\geq0.
	     \end{equation}
	      So, there is one to one correspondence between the {\it state-phase} probabilities of Erlang queue with multiple arrivals and the state probabilities of its queue length process. Therefore, the system of difference-differential equations in Theorem \ref{t3.1} is concisely represented in \eqref{lengthDE}.
	    \begin{remark}
	    	For $l=1$, the system of differential equations given in \eqref{lengthDE} reduces to that of the Erlang queue (see Ascione {\it{et al.}} (2020), Eq. (11)).
	    \end{remark}
	    \begin{remark}
	    	Luchak (1958) studied the queueing process that is described by the  following system of differential equations:
	    	\begin{align}\label{l58de}
	    		\left.
	    		\begin{aligned}
	    			\frac{\mathrm {d}}{\mathrm {d}t}p_{0}(t)
	    			&=-\lambda p_{0}(t)+k\mu p_{1}(t),  \\
	    			\frac{\mathrm {d}}{\mathrm {d}t}p_{n}(t)
	    			&=-(\lambda + k \mu)p_{n}(t)+k\mu p_{n+1}(t)+\lambda\sum_{m=1}^{n}c'_{m}p_{n-m}(t), \, n \geq 1, 
	    		\end{aligned}
	    		\right\} 
	    	\end{align}
	    	with $p_{0}(0)=1$. 
	    	On choosing 
	    	\begin{equation*}
	    		c'_{m}=
	    		\begin{cases}
	    			\lambda_{i}/\sum_{j=1}^{l}\lambda_{j},\,m=ik,\,i=1,2,\dots,l,\\
	    			0,\, \text{otherwise},
	    		\end{cases}
	    	\end{equation*}
	    	the system of differential equations given in \eqref{l58de} reduces to that of $\{\mathcal{L}(t)\}_{t\geq0}$ in \eqref{lengthDE}.
	    \end{remark}
	    Let $\mathcal{M}(t)=\mathbb{E}(\mathcal{L}(t)|\mathcal{L}(0)=0)$ and $\mathcal{M}_{2}(t)=\mathbb{E}(\mathcal{L}(t)^{2}|\mathcal{L}(0)=0)$ be the 
	    mean queue length and the second moment of $\{\mathcal{L}(t)\}_{t\geq0}$, respectively. 
	As there are no customers at time 0, and using \begin{equation*}
		\sum_{m=1}^{\infty}mc_{m}'=k \sum_{i=1}^{l}ic_{i},
	\end{equation*}
	in Luchak (1956), Eq. (36), it follows that $\mathcal{M}(t)$ is the solution of following Cauchy problem:
	\begin{align}\label{mtde}
		\frac{\mathrm{d}}{\mathrm{d}t}\mathcal{M}(t)&=k\Big(\lambda \sum_{i=1}^{l}ic_{i}-\mu\Big)+k\mu p_{0}(t),
	\end{align}
	with $\mathcal{M}(0)=1$. On solving it, we get
	\begin{equation}\label{m1t}
		\mathcal{M}(t)=k\Big(\lambda \sum_{i=1}^{l}ic_{i}-\mu\Big)t+k\mu\int_{0}^{t}p_{0}(y)\mathrm{d}y.
	\end{equation}
   
   Similarly, from Luchak (1958), Eq. (22), the second moment of the queue length process can be obtained in the following form:
	\begin{align*}
		\mathcal{M}_{2}(t)&=k\Big(\mu+k\lambda\sum_{i=1}^{l}i^{2}c_{i}\Big)t+k^{2}\Big(\lambda\sum_{i=1}^{l}ic_{i}-\mu\Big)^{2}t^{2}\nonumber\\
		&\ \ -2k^{2}\mu \Big(\lambda\sum_{i=1}^{l}ic_{i}-\mu\Big)\int_{0}^{t}yp_{0}(y)\mathrm{d}y+\Big(2k^{2}\mu\Big(\lambda\sum_{i=1}^{l}ic_{i}-\mu\Big)t-k\mu\Big)\int_{0}^{t}p_{0}(y)\mathrm{d}y.
	\end{align*}
	
	\begin{remark}
		For $l=1$, the mean queue length given in \eqref{m1t} get reduced to that of the Erlang queue (Ascione {\it{et al.}} (2020), Eq. (12)). 
	\end{remark}
	\subsection{Busy period}
	The busy period of queue is a time interval starting from the arrival of customer(s) into an initially empty system until the system becomes empty again, where the system includes both the server and the queue.
	\begin{theorem}
		Let $B$ denote the duration of busy period in Erlang queue with multiple arrivals, and let $F_{B}(t)$ represent its distribution function. Then, we have
		{\scriptsize \begin{equation}\label{busy}
			F_{B}(t)=a\sum_{n=0}^{\infty}\lambda^{n}\sum_{\substack{\sum_{j=1}^{l}m_{j}=n\\m_{j}\in\mathbb{N}_{0}}}\Big(\prod_{i=1}^{l}\frac{c_{i}^{m_{i}}}{m_{i}!}\Big)\frac{(k\mu)^{a+k\sum_{j=1}^{l}jm_{j}}}{(a+k\sum_{j=1}^{l}jm_{j})!}\int_{0}^{t}z^{a+n+k\sum_{j=1}^{l}jm_{j}-1}e^{-(\lambda+k\mu)z}\mathrm{d}z,\,t\geq0.
		\end{equation}}
	   \end{theorem}
	   \begin{proof}
		In GCP, a packet of jumps (customers) arrives according to the Poisson process (see Dhillon and Kataria (2024a), p. 1112).

		Let us consider the process $\{{\mathcal{L}}_{*}(t)\}_{t\geq0}$  which starts with the arrival of customer(s) and ceases whenever the number of phases in $\{\mathcal{L}(t)\}_{t\geq0}$ becomes zero. Let
	    $p^{*}_{n}(t)=\mathrm{Pr}({\mathcal{L}}_{*}(t)=n|{\mathcal{L}}_{*}(0)=a)$, $a\in \{k, 2k,\dots,lk\}$
		be its state probabilities.
		The state probabilities of $\{\mathcal{L}_{*}(t)\}_{t\geq0}$ solves the following system of differential equations (see Luchak (1958), Eq. (36)-Eq. (38)):
		\begin{align}\label{busyDE}
			\left.
			\begin{aligned}
				\frac{\mathrm {d}}{\mathrm {d}t}p^{*}_{0}(t)
				&=k\mu p^{*}_{1}(t),  \\
				\frac{\mathrm {d}}{\mathrm {d}t}p^{*}_{n}(t)
				&=-(\lambda + k \mu)p^{*}_{n}(t)+k\mu p^{*}_{n+1}(t), \, 1\leq n \leq k, \\
				\frac{\mathrm {d}}{\mathrm {d}t}p^{*}_{n}(t)
				&=-(\lambda + k \mu)p^{*}_{n}(t)+k\mu p^{*}_{n+1}(t)+\lambda\sum_{m=1}^{n-1}c'_{m}p^{*}_{n-m}(t), \, n \geq k+1,
			\end{aligned}
			\right\}
		\end{align}
		with initial conditions $p^{*}_{a}(0)=1$.
		Here, 
		\begin{equation*}
			c'_{m}=
			\begin{cases}
				\lambda_{i}/\sum_{j=1}^{l}\lambda_{j},\,m=ik,i=1,2,\dots,l,\\
				0,\, \text{otherwise}.
			\end{cases}
		\end{equation*}
		The processes $\{{\mathcal{L}}_{*}(t)\}_{t\geq0}$ and $\{\mathcal{L}(t)\}_{t\geq0}$ have the same behavior except that $\{{\mathcal{L}}_{*}(t)\}_{t\geq0}$ starts from $a$ (indicating the state when the first group of customers enter the queue) whereas $\{\mathcal{L}(t)\}_{t\geq0}$ starts from 0. Additionally, 0 is treated as an absorbing state in this process. By the construction of $\{\mathcal{L}_{*}(t)\}_{t\geq0}$, we have $F_{B}(t)=p^{*}_{0}(t)$.
		
		Let us define the generating function of $\{{\mathcal{L}}_{*}(t)\}_{t\geq0}$ as follows:
		\begin{equation*}
			G_{*}(x,t)=\sum_{n=1}^{\infty}x^{n}p^{*}_{n}(t).
		\end{equation*}
		On multiplying the first and second equation of (\ref{busyDE})   by $x^{n+1}$ and summing over $n$, we get
		\begin{equation}\label{bpde}
			x\frac{\partial}{\partial t}G_{*}(x,t)=\Big(k\mu-(\lambda+k\mu)x+\lambda x\sum_{i=1}^{l}c_{i}x^{ik}\Big)G_{*}(x,t)-k\mu xp^{*}_{1}(t),
		\end{equation}
		with $G_{*}(x,0)=x^{a}$.
		On taking the Laplace transform on both sides of (\ref{bpde}), we have
		\begin{equation*}
			\tilde{G}_{*}(x,z)=\frac{x^{a+1}-k\mu x\tilde{\bf{p_{1}}}(z)}{(\lambda + k \mu +z)x-k \mu - \lambda x\sum_{i=1}^{l}c_{i}x^{ik}}, \, z>0.  
		\end{equation*}
		Now, from the same argument as used for \eqref{3.5}, we obtain
		\begin{equation}\label{kmup1}
			k\mu \tilde{\bf{p_{1}}}(z)=\theta^{a},\, a=k, 2k,\dots,lk.
		\end{equation}
		On taking the inverse Laplace transform of \eqref{kmup1} and using (\ref{3.14}), we obtain
		\begin{equation}\label{kmup1t}
			k\mu p^{*}_{1}(t)=\frac{k\mu a}{\tau}\Big(\frac{\tau^{a}}{a!} +\sum_{n=1}^{\infty}\eta^{n}\sum_{\substack{\sum_{j=1}^{l}m_{j}=n \\ m_{j}\in \mathbb{N}_{0}}}\Big({\prod_{i=1}^{l}\frac{c_{i}^{m_{i}}}{m_{i}!}}\Big)\frac{\tau^{a+n+k\sum_{j=1}^{l}jm_{j}}}{(a+k\sum_{j=1}^{l}jm_{j})!}\Big)e^{-(1+\eta)\tau}.
		\end{equation}
		On substituting \eqref{kmup1t} in the first equation of (\ref{busyDE}), we have
		{\small \begin{equation*}
				\frac{\mathrm {d}}{\mathrm {d}t}p^{*}{_{0}}(t)=a\Big(\sum_{n=0}^{\infty}\lambda^{n}\sum_{\substack{\sum_{j=1}^{l}m_{j}=n\\m_{j}\in\mathbb{N}_{0}}}\Big(\prod_{i=1}^{l}\frac{c_{i}^{m_{i}}}{m_{i}!}\Big)\frac{(k\mu)^{a+k\sum_{j=1}^{l}jm_{j}}}{(a+k\sum_{j=1}^{l}jm_{j})!}t^{a+n+k\sum_{j=1}^{l}jm_{j}-1}\Big)e^{-(\lambda+k\mu)t},
		\end{equation*}}
		whose solution gives the required result.		
	\end{proof}
	\begin{remark}
		For $l=1$, the distribution in \eqref{busy} get reduced to that of $M/E_{k}/1$ queue (see Ascione \textit{et al.} (2020), Eq. (13)).
	\end{remark}

	\section{Time-Changed Erlang Queue with Multiple Arrivals}\label{sec4}
    In this section, we introduce an Erlang queue with multiple arrivals that is time-changed by the first hitting time of an independent $\alpha$-stable subordinator. First, we give the definition of stable subordinator and its first hitting time.
    
	An $\alpha$-stable subordinator $\{D_{\alpha}(t)\}_{t\geq0}$, $0<\alpha<1$ is a one dimensional L\'evy process whose Laplace transform is given by (see Applebaum (2009)) $E(e^{-zD_{\alpha}(t)})=e^{-tz^{\alpha}}$, $z>0$. The first hitting time of $\{D_{\alpha}(t)\}_{t\geq0}$, that is, $Y_{\alpha}(t)=\text{inf}\{s\geq0:D_{\alpha}(s)>t\}$, $t\geq0$ is known as the inverse $\alpha$-stable subordinator. Its density function has the following Laplace transform (see Meerschaert {\it{et al.}} (2013)):
		\begin{equation}\label{ltLnu}
			\mathbb{L}(\mathrm{Pr}(Y_{\alpha}(t)\in\mathrm{d}y))(z)=z^{\alpha-1}e^{-yz^{\alpha}}\mathrm{d}y, \, z>0.
	\end{equation}
	
	From Section \ref{sec3}, we recall the {\it state-phase} process $\mathcal{Q}(t)=(\mathcal{N}(t),\mathcal{S}(t))$, $t\geq0$, where $\mathcal{N}(t)$ is the number of customers in the system and $\mathcal{S}(t)$ is the phase of a customer being served at time $t$. 
	
	Let us consider an inverse $\alpha$-stable subordinator $\{Y_{\alpha}(t)\}_{t\geq0}$ which is independent of $\{\mathcal{N}(t)\}_{t\geq0}$ and $\{\mathcal{S}(t)\}_{t\geq0}$. We define a {\it time-changed state-phase} process $\{\mathcal{Q}^{\alpha}(t)\}_{t\geq0}$ as follows: 
	\begin{equation*}
		\mathcal{Q}^{\alpha}(t)=\mathcal{Q}(Y_{\alpha}(t)),\, 0<\alpha<1
	\end{equation*}
	with $\mathcal{Q}^{1}(t)=\mathcal{Q}(t)$. For $t\geq0$, let us denote its state probabilities by
	\begin{align*}
		p_{n,s}^{\alpha}(t)&=\mathrm{Pr}(\mathcal{Q}^{\alpha}(t)=(n,s)|\mathcal{Q}^{\alpha}(0)=(0,0)),\, (n,s)\in \mathcal{H}\\
		p_{0}^{\alpha}(t)&=\mathrm{Pr}(\mathcal{Q}^{\alpha}(t)=(0,0)|\mathcal{Q}^{\alpha}(0)=(0,0)).
	\end{align*}

	In the following result, we obtain the system of fractional differential equations that governs the state probabilities of time-changed Erlang queue with multiple arrivals. The fractional derivative involved is the Caputo fractional derivative defined as follows (see Kilbas {\it{et al.}} (2006)): 
	\begin{equation}\label{cptd}
		\frac{\mathrm{d}^{\alpha}}{\mathrm{d}t^{\alpha}}f(t)=
		\begin{cases}
			\frac{1}{\Gamma(1-\alpha)}\int_{0}^{t}(t-y)^{-\alpha}f'(y)\mathrm{d}y,\, 0<\alpha<1, \\
			f'(y), \, \alpha=1,
		\end{cases}
	\end{equation}
	whose Laplace transform is given by
	\begin{equation}\label{cptlp}
		\mathbb{L}\Big(\frac{\mathrm{d}^{\alpha}}{\mathrm{d}t^{\alpha}}f(t)\Big)(z)=z^{\alpha}\Tilde{f}(z)-z^{\alpha-1}f(0).
	\end{equation} 
	\begin{theorem}\label{cptfde}
		The state probabilities of $\{\mathcal{Q}^{\alpha}(t)\}_{t\geq 0}$ solves the following system of difference-differential equations:
		\begin{align*}
				\frac{\mathrm{d}^{\alpha}}{\mathrm{d}t^{\alpha}}p_{0}^{\alpha}(t)
				&=-\lambda p_{0}^{\alpha}(t)+k\mu p_{1,1}^{\alpha}(t),  \\
				\frac{\mathrm{d}^{\alpha}}{\mathrm{d}t^{\alpha}}p_{1,s}^{\alpha}(t)
				&=-(\lambda + k \mu)p_{1,s}^{\alpha}(t)+k\mu p_{1,s+1}^{\alpha}(t), \, 1 \leq s \leq k-1, \\
				\frac{\mathrm{d}^{\alpha}}{\mathrm{d}t^{\alpha}}p_{1,k}^{\alpha}(t)
				&=-(\lambda + k \mu)p_{1,k}^{\alpha}(t)+k\mu p_{2,1}^{\alpha}(t)+ \lambda c_{1}p_{0}^{\alpha}(t),  \\
				\frac{\mathrm{d}^{\alpha}}{\mathrm{d}t^{\alpha}}p_{n,s}^{\alpha}(t)
				&=-(\lambda + k \mu)p_{n,s}^{\alpha}(t)+k\mu p_{n,s+1}^{\alpha}(t)+\lambda \sum_{m=1}^{\text{min}\{n,l\}}c_{m}p_{n-m,s}^{\alpha}(t), \,  1\leq s \leq k-1,  \, n \geq 2,  \\
				\frac{\mathrm{d}^{\alpha}}{\mathrm{d}t^{\alpha}}p_{n,k}^{\alpha}(t)
				&=-(\lambda + k \mu)p_{n,k}^{\alpha}(t)+k\mu p_{n+1,1}^{\alpha}(t)+\lambda \sum_{m=1}^{n-1}c_{m}p_{n-m,k}^{\alpha}(t)+\lambda c_{n}p_{0}^{\alpha}(t),\, 2\leq n\leq l, \\
				\frac{\mathrm{d}^{\alpha}}{\mathrm{d}t^{\alpha}}p_{n,k}^{\alpha}(t)&=-(\lambda + k \mu)p_{n,k}^{\alpha}(t)+k\mu p_{n+1,1}^{\alpha}(t)+ \lambda \sum_{m=1}^{l}c_{m}p_{n-m,k}^{\alpha}(t),\,n >l,
		\end{align*}
		with $p_{0}^{\alpha}(0)=1$ and
		$p_{n,s}^{\alpha}(0)=0$, $1\leq s\leq k$, $n\geq 1$, where $\frac{\mathrm{d}^{\alpha}}{\mathrm{d}t^{\alpha}}$ is the Caputo fractional derivative defined in \eqref{cptd}.
	\end{theorem}
	\begin{proof}		
		Let us consider the following generating function:
        \begin{equation*}
			G^{\alpha}(x,t)=p_{0}^{\alpha}(t)+\sum_{n=1}^{\infty} \sum_{s=1}^{k}x^{(n-1)k+s}p_{n,s}^{\alpha}(t), \,t\geq0,\,|x|\leq1.
		\end{equation*}
		For $\alpha=1$, it reduces to the generating function of $\mathcal{Q}(t)$, that is, $G^{1}(x,t)=G(x,t)$.
	
		In the system of difference-differential equations given in Theorem \ref{cptfde}, on multiplying the first equation by \(x\), the second by \(x^{s+1}\), the third by \(x^{k+1}\), the fourth by \(x^{k(n-1)+s+1}\), the fifth and sixth by \(x^{nk+1}\), and then summing all these equations over \(n\) and \(s\), it can be shown that the state probabilities of $\{\mathcal{Q}^{\alpha}(t)\}_{t\geq0}$ solves the system of difference-differential equations in Theorem \ref{cptfde} if and only if its generating function is the solution of the following equation:
		\begin{equation}\label{cachyeq}
			x\frac{\mathrm{d}^{\alpha}}{\mathrm{d}t^{\alpha}}G^{\alpha}(x,t)=\Big(k\mu(1-x)-\lambda x\Big(1-\sum_{m=1}^{l}c_{m}x^{mk}\Big)\Big)G^{\alpha}(x,t)-k\mu(1-x)p_{0}^{\alpha}(t),
		\end{equation}
		with initial condition $G^{\alpha}(x,0)=1$. 
		On taking the Laplace transform on both sides of \eqref{cachyeq} and using \eqref{cptlp}, we get 
		\begin{equation}\label{cauchyLL}
			\Big(xz^{\alpha}-k\mu(1-x)+\lambda x\Big(1-\sum_{m=1}^{l}c_{m}x^{mk}\Big)\Big)\tilde{G}^{\alpha}(x,z)=xz^{\alpha-1}-k\mu(1-x)\tilde{p_{0}}^{\alpha}(z).
		\end{equation}
		
		For $(n,s) \in \mathcal{H}$, we have
		\begin{align}\label{pnsalphat}
			p_{n,s}^{\alpha}(t)
			&=\mathrm{Pr}(\mathcal{Q}^{\alpha}(t)=(n,s)|\mathcal{Q}^{\alpha}(0)=(0,0))\nonumber\\
			&=\mathrm{Pr}(\mathcal{Q}(Y_{\alpha}(t))=(n,s)|\mathcal{Q}(0)=(0,0))\nonumber\\
			&=\int_{0}^{\infty}\mathrm{Pr}(\mathcal{Q}(y)=(n,s)|\mathcal{Q}(0)=(0,0))\mathrm{Pr}(Y_{\alpha}(t)\in\mathrm{d}y)\nonumber\\
			&=\int_{0}^{\infty}p_{n,s}(y)\mathrm{Pr}(Y_{\alpha}(t)\in\mathrm{d}y)
		\end{align}
		and
		\begin{equation}\label{p0mu}
			{p}_{0}^{\alpha}(t)=\int_{0}^{\infty}p_{0}(y)\mathrm{Pr}(Y_{\alpha}(t)\in\mathrm{d}y).
		\end{equation}
		Thus, we have 
		\begin{equation}\label{Gxtmu}
			{G}^{\alpha}(x,t)=\int_{0}^{\infty}G(x,y)\mathrm{Pr}(Y_{\alpha}(t)\in\mathrm{d}y).
		\end{equation}
		On taking the Laplace transform of \eqref{p0mu} and \eqref{Gxtmu}, we get
		\begin{equation}\label{p0nul}
			\Tilde{p_{0}}^{\alpha}(z)=z^{\alpha-1}\int_{0}^{\infty}p_{0}(y)e^{-yz^{\alpha}}\mathrm{d}y
		\end{equation}
		and
		\begin{equation}\label{Gxtnul}
			\Tilde{G}^{\alpha}(x,z)=z^{\alpha-1}\int_{0}^{\infty}G(x,y)(y)e^{-yz^{\alpha}}\mathrm{d}y,
		\end{equation}
		respectively, where we have used \eqref{ltLnu}. 
		
		Now, by using \eqref{p0nul} and \eqref{Gxtnul} in \eqref{cauchyLL}, we obtain \begin{equation*}\label{lptcde}
				\int_{0}^{\infty}\Big(\Big(xz^{\alpha}-k\mu(1-x)+\lambda x\Big(1-\sum_{m=1}^{l}c_{m}x^{mk}\Big)\Big)G(x,y)+k\mu(1-x)p_{0}(y)\Big)z^{\alpha-1}e^{-yz^{\alpha}}\mathrm{d}y=xz^{\alpha-1},
		\end{equation*}
		which on using \eqref{3.2} gives us the following identity: 
		\begin{equation*}\label{integral}
			\int_{0}^{\infty}\Big(xz^{\alpha}G(x,y)-x \frac{\partial }{\partial y}G(x,y)\Big)z^{\alpha-1}e^{-yz^{\alpha}}\mathrm{d}y=xz^{\alpha-1}.
		\end{equation*}
		This establishes the required result.
	\end{proof}
 \begin{remark}
		For $l=1$, the system of difference-differential equations given in Theorem \ref{cptfde} reduces to that of fractional Erlang queue (see Griffiths \textit{et al.} (2006), Eq. (4)), and for $\alpha=1$, it reduces to the system of difference-differential equations for the {\it{state-phase}} probabilities of Erlang queue with multiple arrivals (see Theorem \ref{t3.1}). 
	\end{remark} 
Now, we obtain the probability that there are no customers in the system at time $t$. The following result will be used:

 The Laplace transform of three-parameter Mittag-Leffler function is given by (see Haubold {\it et al.} (2011)) 
\begin{equation}\label{ltm}
	\mathbb{L}(t^{\beta-1}E_{\alpha,\beta}^{\gamma}(\omega t^{\alpha}))(z)=\frac{z^{\alpha\gamma-\beta}}{(z^{\alpha}-\omega)^{\gamma}}, t,\omega\in\mathbb{R},\,\alpha,\,\beta,\,\gamma,\,z>0,\,|\omega z^{\alpha}|<1,
\end{equation}
where
 	\begin{equation}\label{Mittag}
		E_{\alpha,\beta}^{\gamma}(t)\coloneqq\sum_{r=0}^{\infty}\frac{\Gamma(\gamma+r)t^{r}}{r!\Gamma(\gamma)\Gamma(\alpha r+\beta)}
	\end{equation}
	is the three-parameter Mittag-Leffler function (see Kilbas {\it et al.} (2006)).
	\begin{theorem}\label{theorem4.2}
		Let $ \textbf{m}=(m_{1},m_{2},\dots,m_{l})$, $m_{j}\in \mathbb{N}_{0}$. The state probability $p_{0}^{\alpha}(t)$ of $\{\mathcal{Q}^{\alpha}(t)\}_{t\geq0}$ is given by
		\begin{equation*}
			p_{0}^{\alpha}(t)=\sum_{h=1}^{\infty}\sum_{n=0}^{\infty}\sum_{\substack{\sum_{j=1}^{l}m_{j}=n\\m_{j}\geq0}}C_{h,n}^{0}(\textbf{m})t^{\beta_{h,n}^{0}(\textbf{m})-1}E_{\alpha,\beta_{h,n}^{0}(\textbf{m})}^{\gamma_{h,n}^{0}(\textbf{m})}(-(\lambda+k\mu) t^{\alpha}),
		\end{equation*}
		where $E_{\alpha,\beta}^{\gamma}(\cdot)$ is the three parameter Mittag-Leffler function defined in \eqref{Mittag} and
		\begin{equation*}
			C_{h,n}^{0}(\textbf{m})=h\lambda^{n}\Big(\prod_{i=1}^{l}\frac{c_{i}^{m_{i}}}{m_{i}!}\Big)\frac{(k\mu)^{\gamma_{h,n}^{0}(\textbf{m})-n-1}}{(\gamma_{h,n}^{0}(\textbf{m})-n)!}(\gamma_{h,n}^{0}(\textbf{m})-1)!,
		\end{equation*}
		$\gamma_{h,n}^{0}(\textbf{m})=h+n+k\sum_{j=1}^{l}jm_{j}$ and $\beta_{h,n}^{0}(\textbf{m})=\alpha(\gamma_{h,n}^{0}(\textbf{m})-1)+1$.
	\end{theorem}
	\begin{proof}
		Using \eqref{3.15} in \eqref{p0nul}, we get 
		\begin{align}
			\Tilde{{p}_{0}}^{\alpha}(z)
			&=\sum_{h=1}^{\infty}\sum_{n=0}^{\infty}\sum_{\substack{\sum_{j=1}^{l}m_{j}=n\\m_{j}\geq0}}h\lambda^{n}\Big(\prod_{i=1}^{l}\frac{c_{i}^{m_{i}}}{m_{i}!}\Big)\frac{(k\mu)^{h-1+k\sum_{j=1}^{l}jm_{j}}}{(h+k\sum_{j=1}^{l}jm_{j})!}z^{\alpha-1}\nonumber\\
			&\hspace{7.3cm}\cdot\int_{0}^{\infty}y^{h+n+k\sum_{j=1}^{l}jm_{j}-1}e^{-(\lambda+k\mu+z^{\alpha})y}\mathrm{d}y\nonumber\\
			&=\sum_{h=1}^{\infty}\sum_{n=0}^{\infty}\sum_{\substack{\sum_{j=1}^{l}m_{j}=n\\m_{j}\geq0}}h\lambda^{n}\Big(\prod_{i=1}^{l}\frac{c_{i}^{m_{i}}}{m_{i}!}\Big)\frac{(k\mu)^{h-1+k\sum_{j=1}^{l}jm_{j}}}{(h+k\sum_{j=1}^{l}jm_{j})!}\Big(h+n-1+k\sum_{j=1}^{l}jm_{j}\Big)!\nonumber\\
			&\hspace{9cm}\frac{z^{\alpha-1}}{(\lambda+k\mu+z^{\alpha})^{h+n+k\sum_{j=1}^{l}jm_{j}}}\nonumber\\
			&=\sum_{h=1}^{\infty}\sum_{n=0}^{\infty}\sum_{\substack{\sum_{j=1}^{l}m_{j}=n\\m_{j}\geq0}}C_{h,n}^{0}(\textbf{m})\frac{z^{\alpha\gamma_{h,n}^{0}(\textbf{m})-\beta_{h,n}^{0}(\textbf{m})}}{(\lambda+k\mu+z^{\alpha})^{\gamma_{h,n}^{0}(\textbf{m})}}\label{p0alpha4}.
		\end{align}	
		On taking the inverse Laplace transform of \eqref{p0alpha4} and using \eqref{ltm}, we get the required result.
	\end{proof}
    \begin{remark}
		On taking $l=1$ in Theorem \ref{theorem4.2}, we get the corresponding result for fractional Erlang queue (see Ascione {\it{et al.}} (2020), Theorem 5.1), and for $\alpha=1$, we get the zero {\it state-phase} probability of $\{\mathcal{Q}(t)\}_{t\geq0}$ (see \eqref{3.15}).
	\end{remark}
	The following result will be used.
	\begin{lemma}\label{lemma4.1}
		Let $r$ be a non-negative integer. Then,
		{\footnotesize\begin{equation*}
			\int_{0}^{\infty}\int_{0}^{y}p_{0}(u)(y-u)^{r}e^{-(\lambda+k\mu)(y-u)}e^{-yz^{\alpha}}\mathrm{d}u\mathrm{d}y=\sum_{h=1}^{\infty}\sum_{n=0}^{\infty}\sum_{\substack{\sum_{j=1}^{l}m_{j}=n\\m_{j}\geq0}}C_{h,n}^{0}(\textbf{m})\frac{r!}{(\lambda+k\mu+z^{\alpha})^{\gamma_{h,n}^{0}(\textbf{m})+r+1}},
		\end{equation*}}
	where $\textbf{m}=(m_{1},m_{2},\dots,m_{l})$ and $p_{0}(u)$ is the zero {\it state-phase} probability of $\{\mathcal{Q}(t)\}_{t\geq0}$.
	\end{lemma}
	\begin{proof}
	Let $\chi_{A}(\cdot)$ be an indicator function on set $A$. Then for $r\geq0$, we have
	{\small	\begin{align}
			\int_{0}^{\infty}&\int_{0}^{y}p_{0}(u)(y-u)^{r}e^{-(\lambda+k\mu)(y-u)}e^{-yz^{\alpha}}\mathrm{d}u\mathrm{d}y\nonumber\\
			&=\int_{0}^{\infty}\int_{0}^{\infty}p_{0}(u)(y-u)^{r}e^{-(\lambda+k\mu+z^{\alpha})(y-u)}e^{-uz^{\alpha}}\chi_{[0,y]}(u)\mathrm{d}u\mathrm{d}y\nonumber\\
			&=\int_{0}^{\infty}\int_{0}^{\infty}p_{0}(u)(y-u)^{r}e^{-(\lambda+k\mu+z^{\alpha})(y-u)}e^{-uz^{\alpha}}\chi_{[0,y]}(u)\mathrm{d}y\mathrm{d}u\,\,(\text{using Fubini's theorem})\nonumber\\
			&=\int_{0}^{\infty}\int_{0}^{\infty}p_{0}(u)x^{r}e^{-(\lambda+k\mu+z^{\alpha})x}e^{-uz^{\alpha}}\chi_{[0,x+u]}(u)\mathrm{d}x\mathrm{d}u\nonumber\\
			&=\int_{0}^{\infty}\int_{0}^{\infty}p_{0}(u)x^{r}e^{-(\lambda+k\mu+z^{\alpha})x}e^{-uz^{\alpha}}\chi_{[0,\infty)}(x)\mathrm{d}x\mathrm{d}u\,\,\text{(as $\chi_{[0,x+u]}(u)=\chi_{[0,\infty)}(x)$)}\nonumber\\
			&=\Big(\int_{0}^{\infty}p_{0}(u)e^{-uz^{\alpha}}\mathrm{d}u\Big)\Big(\int_{0}^{\infty}x^{r}e^{-(\lambda+k\mu+z^{\alpha})x}\mathrm{d}x\Big)\label{int1}.
		\end{align}}
	From \eqref{p0nul}, we have
		\begin{align}
			\int_{0}^{\infty}p_{0}(u)e^{-uz^{\alpha}}\mathrm{d}u&=\frac{\tilde{p_{0}}^{\alpha}(z)}{z^{\alpha-1}}\nonumber\\
			&=\sum_{h=1}^{\infty}\sum_{n=0}^{\infty}\sum_{\substack{\sum_{j=1}^{l}m_{j}=n\\m_{j}\geq0}}\frac{C_{h,n}^{0}(\textbf{m})}{(\lambda+k\mu+z^{\alpha})^{\gamma_{h,n}^{0}(\textbf{m})}},\label{int3}
		\end{align}
		where we have used \eqref{p0alpha4}. On substituting \eqref{int3} in \eqref{int1}, we get the required result.	\end{proof}
		\begin{remark}
			For $l=1$, Lemma \ref{lemma4.1} get reduced to Lemma 5.4 of Ascione {\it{et al.}} (2020).
		\end{remark}
	\begin{theorem}\label{thm4.3}
		Let $ \textbf{m}=(m_{1},m_{2},\dots,m_{l})$, and $\textbf{m}'=(m'_{1},m'_{2},\dots,m'_{l})$, $m_{j}, m'_{j}\in \mathbb{N}_{0}$.
         For $t\geq0$ and $1\leq s \leq k$, the state probabilities of $\{\mathcal{Q}^{\alpha}(t)\}_{t\geq0}$ are given by		
         {\footnotesize \begin{align*}
         	p_{n,s}^{\alpha}(t)&=\sum_{\substack{m_{j},r\geq0\\\sum_{j=1}^{l}jm_{j} -r=n}}A_{r}^{n,s}(\textbf{m})t^{\pi_{r}^{n,s}(\textbf{m})-1}E_{\alpha,\pi_{r}^{n,s}(\textbf{m})}^{a_{r}^{n,s}(\textbf{m})}(-(\lambda+k\mu) t^{\alpha})\\
         	&\ \ +\sum_{\substack{m_{j},r\geq0\\\sum_{j=1}^{l}jm_{j} -r=n}}\sum_{h=1}^{\infty}\sum_{w=0}^{\infty}\sum_{\substack{\sum_{j=1}^{l}m'_{j}=w\\m'_{j}\geq0}}B_{r,h,w}^{n,s}(\textbf{m},\textbf{m}')t^{\rho_{r,h,w}^{n,s}(\textbf{m})-1}E_{\alpha,\rho_{r,h,w}^{n,s}(\textbf{m},\textbf{m}')}^{b_{r,h,w}^{n,s}(\textbf{m},\textbf{m}')}(-(\lambda+k\mu) t^{\alpha})\\
         	&\ \ -\sum_{\substack{m_{j},r\geq0\\\sum_{j=1}^{l}jm_{j} -r=n}}\sum_{h=1}^{\infty}\sum_{w=0}^{\infty}\sum_{\substack{\sum_{j=1}^{l}m'_{j}=w\\m'_{j}\geq0}}C_{r,h,w}^{n,s}(\textbf{m},\textbf{m}')t^{\delta_{r,h,w}^{n,s}(\textbf{m},\textbf{m}')-1}E_{\alpha,\delta_{r,h,w}^{n,s}(\textbf{m},\textbf{m}')}^{c_{r,h,w}^{n,s}(\textbf{m},\textbf{m}')}(-(\lambda+k\mu) t^{\alpha}),\,n\geq1.
         	\end{align*}}
    	Here, $E_{\alpha,\beta}^{\gamma}(\cdot)$ is the three parameter Mittag-Leffler function defined in \eqref{Mittag} and
         	\begin{align*}
         		A_{r}^{n,s}(\textbf{m})&=\Big(\prod_{i=1}^{l}\frac{c_{i}^{m_{i}}}{m_{i}!}\Big)\frac{\lambda^{\sum_{j=1}^{l}m_{j}}(k\mu )^{k(r+1)-s}}{(k(r+1)-s)!}(a_{r}^{n,s}(\textbf{m})-1)!,\\
         		C_{r,h,w}^{n,s}(\textbf{m},\textbf{m}')&=\begin{cases}
         			k\mu A_{r}^{n,s+1}(\textbf{m})C_{h,w}^{0}(\textbf{m}'),\,s=1,2,\dots,k-1,\\
         			k\mu A_{r}^{n,1}(\textbf{m}'')C_{h,w}^{0}(\textbf{m}'),\, s=k,
         		\end{cases}	\\
         		c_{r,h,w}^{n,s}(\textbf{m},\textbf{m}')&=\begin{cases}
         			a_{r}^{n,s+1}(\textbf{m})+\gamma_{h,w}^{0}(\textbf{m}'),\,s=1,2,\dots,k-1,\\
         			a_{r}^{n,1}(\textbf{m}'')+\gamma_{h,w}^{0}(\textbf{m}'),\, s=k,
         		\end{cases}
         	\end{align*}
         	where $\textbf{m}''=(m_{1}+1,m_{2},\dots,m_{l})$, 
         	$a_{r}^{n,s}(\textbf{m})=\sum_{j=1}^{l}m_{j}+k(r+1)-s+1$, $\pi_{r}^{n,s}(\textbf{m})=\alpha(a_{r}^{n,s}(\textbf{m})-1)+1$, $B_{r,h,w}^{n,s}(\textbf{m},\textbf{m}')=k\mu A_{r}^{n,s}(\textbf{m})C_{h,w}^{0}(\textbf{m}')$, $b_{r,h,w}^{n,s}(\textbf{m},\textbf{m}')=a_{r}^{n,s}(\textbf{m})+\gamma_{h,w}^{0}(\textbf{m}')$, $\rho_{r,h,w}^{n,s}(\textbf{m},\textbf{m}')=\alpha(b_{r,h,w}^{n,s}(\textbf{m},\textbf{m}')-1)+1$ and 
         	 $\delta_{r,h,w}^{n,s}=\alpha(c_{r,h,w}^{n,s}(\textbf{m},\textbf{m}')-1)+1$.
	\end{theorem}
	\begin{proof}
		 On taking the Laplace transform of \eqref{pnsalphat} and using \eqref{ltLnu}, we get
		\begin{equation}\label{lppnst}
			\Tilde{p}_{n,s}^{\alpha}(z)
			=z^{\alpha-1}\int_{0}^{\infty}p_{n,s}(y)e^{-yz^{\alpha}}\mathrm{d}y.
		\end{equation}
		For $1\leq s\leq k-1$, by using \eqref{pnst}, we have
   {\scriptsize	\begin{align}
			\Tilde{p}_{n,s}^{\alpha}(z)
			&=\sum_{\substack{m_{j},r\geq0\\\sum_{j=1}^{l}jm_{j} -r=n}}\Big(\prod_{i=1}^{l}\frac{c_{i}^{m_{i}}}{m_{i}!}\Big)\frac{\lambda^{\sum_{j=1}^{l}m_{j}}(k\mu )^{k(r+1)-s}}{(k(r+1)-s)!} z^{\alpha-1}\int_{0}^{\infty}y^{\sum_{j=1}^{l}m_{j}+k(r+1)-s}e^{-(\lambda+k\mu+z^{\alpha})y}\mathrm{d}y\nonumber\\
			&\ \  +\sum_{\substack{m_{j},r\geq0\\\sum_{j=1}^{l}jm_{j} -r=n}}\Big(\prod_{i=1}^{l}\frac{c_{i}^{m_{i}}}{m_{i}!}\Big)\frac{\lambda^{\sum_{j=1}^{l}m_{j}}(k\mu )^{k(r+1)-s+1}}{(k(r+1)-s)!} z^{\alpha-1}\int_{0}^{\infty}\int_{0}^{y}p_{0}(u)(y-u)^{\sum_{j=1}^{l}m_{j}+k(r+1)-s}\nonumber\\
			&\hspace{2.5cm} \cdot e^{-(\lambda+k\mu)(y-u)} e^{-yz^{\alpha}}\mathrm{d}u\mathrm{d}y -\sum_{\substack{m_{j},r\geq0\\\sum_{j=1}^{l}jm_{j} -r=n}}\Big(\prod_{i=1}^{l}\frac{c_{i}^{m_{i}}}{m_{i}!}\Big)\frac{\lambda^{\sum_{j=1}^{l}m_{j}}(k\mu )^{k(r+1)-s}}{(k(r+1)-s-1)!} z^{\alpha-1}\nonumber\\
			&\hspace{5.5cm}\cdot\int_{0}^{\infty}\int_{0}^{y}p_{0}(u)(y-u)^{\sum_{j=1}^{l}m_{j}+k(r+1)-s-1}e^{-(\lambda+k\mu)(y-u)}e^{-yz^{\alpha}}\mathrm{d}u\mathrm{d}y\,\,\nonumber\\
			&=\sum_{\substack{m_{j},r\geq0\\\sum_{j=1}^{l}jm_{j} -r=n}}A_{r}^{n,s}(\textbf{m})\frac{z^{\alpha a_{r}^{n,s}(\textbf{m})-\pi_{r}^{n,s}(\textbf{m})}}{(\lambda+k\mu+z^{\alpha})^{a_{r}^{n,s}(\textbf{m})}}+\sum_{\substack{m_{j},r\geq0\\\sum_{j=1}^{l}jm_{j} -r=n}}\sum_{h=1}^{\infty}\sum_{w=0}^{\infty}\sum_{\substack{\sum_{j=1}^{l}m'_{j}=w\\m'_{j}\geq0}}B_{r,h,w}^{n,s}(\textbf{m},\textbf{m}')\nonumber\\
			&\ \ \cdot\frac{z^{\alpha b_{r,h,w}^{n,s}(\textbf{m},\textbf{m}')-\rho_{r,h,w}^{n,s}(\textbf{m},\textbf{m}')}}{(\lambda+k\mu+z^{\alpha})^{b_{r,h,w}^{n,s}(\textbf{m},\textbf{m}')}} -\sum_{\substack{m_{j},r\geq0\\\sum_{j=1}^{l}jm_{j} -r=n}}\sum_{h=1}^{\infty}\sum_{w=0}^{\infty}\sum_{\substack{\sum_{j=1}^{l}m'_{j}=w\\m'_{j}\geq0}}C_{r,h,w}^{n,s}(\textbf{m},\textbf{m}')\frac{z^{\alpha c_{r,h,w}^{n,s}(\textbf{m},\textbf{m}')-\delta_{r,h,w}^{n,s}(\textbf{m},\textbf{m}')}}{(\lambda+k\mu+z^{\alpha})^{c_{r,h,w}^{n,s}(\textbf{m},\textbf{m}')}},\label{pnsalpha}
		\end{align}}where the last step follows by using Lemma \ref{lemma4.1}.
	 
	   Similarly, for $s=k$, from \eqref{lppnst} and using \eqref{pnkt}, we have
	   {\scriptsize \begin{align}
			\Tilde{p}_{n,k}^{\alpha}(z)
			&=\sum_{\substack{m_{j},r\geq0\\\sum_{j=1}^{l}jm_{j} -r=n}}\Big(\prod_{i=1}^{l}\frac{c_{i}^{m_{i}}}{m_{i}!}\Big)\frac{\lambda^{\sum_{j=1}^{l}m_{j}}(k\mu )^{rk}}{(rk)!} z^{\alpha-1}\int_{0}^{\infty}y^{\sum_{j=1}^{l}m_{j}+rk}e^{-(\lambda+k\mu+z^{\alpha})y}\mathrm{d}y\nonumber\\
			&\ \ +\sum_{\substack{m_{j},r\geq0\\\sum_{j=1}^{l}jm_{j} -r=n}}\Big(\prod_{i=1}^{l}\frac{c_{i}^{m_{i}}}{m_{i}!}\Big)\frac{\lambda^{\sum_{j=1}^{l}m_{j}}(k\mu )^{rk+1}}{(rk)!} z^{\alpha-1}\int_{0}^{\infty}\int_{0}^{y}p_{0}(u)(y-u)^{\sum_{j=1}^{l}m_{j}+rk}\nonumber\\
			&\hspace{1.5cm} \cdot e^{-(\lambda+k\mu)(y-u)}e^{-yz^{\alpha}}\mathrm{d}u\mathrm{d}y-\sum_{\substack{m_{j},r\geq0\\\sum_{j=1}^{l}jm_{j} -r=n}}\Big(\frac{c_{1}^{m_{1}+1}}{(m_{1}+1)!}\Big)\Big(\prod_{i=2}^{l}\frac{c_{i}^{m_{i}}}{m_{i}!}\Big)\frac{\lambda^{\sum_{j=1}^{l}m_{j}+1}(k\mu )^{k(r+1)}}{((r+1)k-1)!} z^{\alpha-1}\nonumber\\
			&\hspace{6cm}\cdot\int_{0}^{\infty}\int_{0}^{y}p_{0}(u)(y-u)^{\sum_{j=1}^{l}m_{j}+k(r+1)}e^{-(\lambda+k\mu)(y-u)}e^{-yz^{\alpha}}\mathrm{d}u\mathrm{d}y\,\,\nonumber\\
			&=\sum_{\substack{m_{j},r\geq0\\\sum_{j=1}^{l}jm_{j} -r=n}}A_{r}^{n,k}(\textbf{m})\frac{z^{\alpha a_{r}^{n,k}(\textbf{m})-\pi_{r}^{n,k}(\textbf{m})}}{(\lambda+k\mu+z^{\alpha})^{a_{r}^{n,k}(\textbf{m})}}+\sum_{\substack{m_{j},r\geq0\\\sum_{j=1}^{l}jm_{j} -r=n}}\sum_{h=1}^{\infty}\sum_{w=0}^{\infty}\sum_{\substack{\sum_{j=1}^{l}m'_{j}=w\\m'_{j}\geq0}}B_{r,h,w}^{n,k}(\textbf{m},\textbf{m}')\nonumber\\
			&\ \ \cdot\frac{z^{\alpha b_{r,h,w}^{n,k}(\textbf{m},\textbf{m}')-\rho_{r,h,w}^{n,k}(\textbf{m},\textbf{m}')}}{(\lambda+k\mu+z^{\alpha})^{b_{r,h,w}^{n,k}(\textbf{m},\textbf{m}')}} -\sum_{\substack{m_{j},r\geq0\\\sum_{j=1}^{l}jm_{j} -r=n}}\sum_{h=1}^{\infty}\sum_{w=0}^{\infty}\sum_{\substack{\sum_{j=1}^{l}m'_{j}=w\\m'_{j}\geq0}}C_{r,h,w}^{n,k}(\textbf{m},\textbf{m}')\frac{z^{\alpha c_{r,h,w}^{n,k}(\textbf{m},\textbf{m}')-\delta_{r,h,w}^{n,k}(\textbf{m},\textbf{m}')}}{(\lambda+k\mu+z^{\alpha})^{c_{r,h,w}^{n,k}(\textbf{m},\textbf{m}')}}\label{pnkalpha}.
	\end{align}}On taking the inverse Laplace transform on both side of \eqref{pnsalpha} and \eqref{pnkalpha}, we get the required result.
    \end{proof}
    \begin{remark}
    	For $l=1$, Theorem \ref{thm4.3} reduces to Theorem 5.5 of Ascione {\it{et al.}} (2020). For $\alpha=1$, it reduces to Theorem \ref{t3.2} which can be shown by using \eqref{3.15} in \eqref{pnst} and \eqref{pnkt}.
    \end{remark}
    \subsection{Time-changed queue length process}
	Let us define a time-changed queue length process as $\mathcal{L}^{\alpha}(t)=g_{k}(\mathcal{Q}^{\alpha}(t))$, $t\geq0$. That is,
	\begin{equation*}
		\mathcal{L}^{\alpha}(t)=\left\{
		\begin{array}{ll}
			k(\mathcal{N}^{\alpha}(t)-1)+s,\, \mathcal{N}^{\alpha}(t)>0,\\
			0,\, \mathcal{N}(t)=0,
		\end{array}
		\right.  
	\end{equation*} 
	where $g_{k}(\cdot)$ is as defined in \eqref{gk}. We denote its state probabilities by
	\begin{equation*}
		p_{n}^{\alpha}(t)=\mathrm{Pr}(\mathcal{L}^{\alpha}(t)=n|\mathcal{L}^{\alpha}(0)=0), \, n\geq0,
	\end{equation*}
	such that $	p_{-n}^{\alpha}(t)=0$, $n\geq1$.
	    
	    \begin{theorem}
	    	The state probabilities of $\{\mathcal{L}^{\alpha}(t)\}_{t\geq0}$ solve the following fractional Cauchy problem:
	    	\begin{align}\label{cptlength}
	    		\left.
	    		\begin{aligned}
	    			\frac{\mathrm{d}^{\alpha}}{\mathrm{d}t^{\alpha}}p_{0}^{\alpha}(t)
	    			&=-\lambda p_{0}^{\alpha}(t)+k\mu p_{1}^{\alpha}(t), \\
	    			\frac{\mathrm{d}^{\alpha}}{\mathrm{d}t^{\alpha}}p_{n}^{\alpha}(t)
	    			&=-(\lambda + k \mu)p_{n}^{\alpha}(t)+k\mu p_{n+1}^{\alpha}(t)+\lambda\sum_{m=1}^{n}c'_{m}p_{n-m}^{\alpha}(t), \, n \geq 1,  
	    		\end{aligned}
	    		\right\}
	    	\end{align}
	    	where 
	    	\begin{equation*}
	    		c'_{m}=
	    		\begin{cases}
	    			\lambda_{i}/\sum_{j=1}^{l}\lambda_{j},\,m=ik,i=1,2,\dots,l,\\
	    			0,\, \text{otherwise},
	    		\end{cases}
	    	\end{equation*}
	    	with $p_{0}^{\alpha}(0)=1$ and $p_{n}^{\alpha}(0)=0,\,n\geq1$.
	    	\end{theorem}
    	\begin{proof}
    		From \eqref{gk} and \eqref{gkn}, and using Theorem \ref{cptfde}, we obtain
    		{\small\begin{align*}
    			\frac{\mathrm{d}^{\alpha}}{\mathrm{d}t^{\alpha}}p_{0}^{\alpha}(t)
    			&=-\lambda p_{0}^{\alpha}(t)+k\mu p_{1}^{\alpha}(t),\\
    			\frac{\mathrm{d}^{\alpha}}{\mathrm{d}t^{\alpha}}p_{n}^{\alpha}(t)
    			&=-(\lambda+k\mu)p_{n}^{\alpha}(t)+k\mu p_{n+1}^{\alpha}(t),\,1\leq n\leq k-1,\\
    			\frac{\mathrm{d}^{\alpha}}{\mathrm{d}t^{\alpha}}p_{k}^{\alpha}(t)
    			&=-(\lambda+k\mu)p_{k}^{\alpha}(t)+k\mu p_{k+1}^{\alpha}(t)+\lambda c_{1}p_{0}^{\alpha}(t), \\
    			\frac{\mathrm{d}^{\alpha}}{\mathrm{d}t^{\alpha}}p_{n}^{\alpha}(t)
    			&=-(\lambda+k\mu)p_{n}^{\alpha}(t)+k\mu p_{n+1}^{\alpha}(t)+\lambda\sum_{i=1}^{\text{min}\{n,l\}}c_{i}p_{m-ik}^{\alpha}(t),\,n=k(m-1)+s,\,1\leq s\leq k-1,  	\\
    				\frac{\mathrm{d}^{\alpha}}{\mathrm{d}t^{\alpha}}p_{n}^{\alpha}(t)
    			&=-(\lambda+k\mu)p_{n}^{\alpha}(t)+k\mu p_{n+1}^{\alpha}(t)+\lambda\sum_{i=1}^{\text{min}\{n,l\}}c_{i}p_{n-ik}^{\alpha}(t),\,n=mk.		
    		\end{align*}}
    	As $p_{-n}^{\alpha}(t)=0$, $n\geq1$, we get the required result.
    		\end{proof} 
	\begin{remark}
		For $l=1$, the system of differential equations \eqref{cptlength} reduces to that of the fractional Erlang queue (see Ascione {\it{et al.}} (2020), Eq. (28)), and for $\alpha=1$, it reduces to \eqref{lengthDE}.
	\end{remark}

	\begin{theorem}
		The mean queue length $\mathcal{M}^{\alpha}(t)=\mathds{E}(\mathcal{L}^{\alpha}(t)=n|\mathcal{L}^{\alpha}(0)=0)$ of time-changed Erlang queue with multiple arrivals is the solution of following fractional Cauchy problem:
		\begin{equation}\label{cpdtmalphat}
			\frac{\mathrm{d}^{\alpha}}{\mathrm{d}t^{\alpha}}\mathcal{M}^{\alpha}(t)=k\mu p_{0}^{\alpha}(t)-k\mu+\lambda\sum_{i=1}^{l}ic_{i},\,
			\mathcal{M}^{\alpha}(0)=0
		\end{equation}
	\end{theorem}
	\begin{proof}
		By the definition of mean queue length, we have
		\begin{equation}\label{dtmt}
			\mathcal{M}^{\alpha}(t)=\sum_{n=0}^{\infty}np_{n}^{\alpha}(t).
		\end{equation}
		As $p_{0}^{\alpha}(0)=1$, we have $\mathcal{M}^{\alpha}(0)=0$.
		On taking the Caputo fractional derivative of order $\alpha$ on both sides of \eqref{dtmt} and using the system of differential equations \eqref{cptlength}, we get
		\begin{align}
			\frac{\mathrm{d}^{\alpha}}{\mathrm{d}t^{\alpha}}\mathcal{M}^{\alpha}(t)
			&=-(\lambda + k \mu)\mathcal{M}^{\alpha}(t)+k\mu\sum_{n=1}^{\infty}np_{n+1}^{\alpha}(t)+\lambda\sum_{n=1}^{\infty}n\sum_{m=1}^{n}c'_{m}p_{n-m}^{\alpha}(t).\label{malphat}
		\end{align}
		Note that
		\begin{equation}\label{sum1}
			\sum_{n=1}^{\infty}np_{n+1}^{\alpha}(t)=\mathcal{M}^{\alpha}(t)+p_{0}^{\alpha}(t)-1
		\end{equation}
		and
		\begin{align}
			\sum_{n=1}^{\infty}n\sum_{m=1}^{n}c'_{m}p_{n-m}^{\alpha}(t)
			&=\sum_{m=1}^{\infty}c_{m}'\sum_{n=m}^{\infty}np_{n-m}^{\alpha}(t)\nonumber\\
			&=\sum_{m=1}^{\infty}c_{m}'(\mathcal{M}^{\alpha}(t)+m)\nonumber\\
			&=\mathcal{M}^{\alpha}(t)+k\sum_{i=1}^{l}ic_{i}\label{sum2}.
		\end{align}
		On substituting \eqref{sum1} and \eqref{sum2} in \eqref{malphat}, we get the required result.
	\end{proof}
	\begin{remark}
		For $l=1$, the fractional Cauchy problem in \eqref{cpdtmalphat} reduces to that of the fractional Erlang queue (see Ascione {\it{et al.}} (2020), Eq. (42)), and for $\alpha=1$, it reduces to \eqref{mtde}.
	\end{remark}
	\begin{theorem}
		The mean queue length $\mathcal{M}^{\alpha}(t)$ of time-changed Erlang queue with multiple arrivals $\{\mathcal{Q}^{\alpha}(t)\}_{t\geq0}$, $0<\alpha<1$ is given by
		\begin{align}\label{malpahat}
			\mathcal{M}^{\alpha}(t)&=k\Big(\lambda\sum_{i=1}^{l}ic_{i}-\mu\Big)\frac{t^{\alpha}}{\Gamma({\alpha+1})}\nonumber\\
			&\ \ +k\mu\sum_{h=1}^{\infty}\sum_{n=0}^{\infty}\sum_{\substack{\sum_{j=1}^{l}m_{j}=n\\m_{j}\geq0}}C_{h,n}^{0}(\textbf{m})t^{\beta_{h,n}^{\mathcal{M}}(\textbf{m})-1}E_{\alpha,\beta_{h,n}^{\mathcal{M}}(\textbf{m})}^{\gamma_{h,n}^{0}(\textbf{m})}(-(\lambda+k\mu) t^{\alpha}),\,t\geq0,
		\end{align}
	where $\textbf{m}=(m_{1}, m_{2},\dots,m_{l})$ and $\beta_{h,n}^{\mathcal{M}}(\textbf{m})=\alpha\gamma_{h,n}^{0}(\textbf{m})+1$.
	\end{theorem}
	\begin{proof}
		From \eqref{gkn} and using \eqref{pnsalphat} in \eqref{dtmt}, we have
		\begin{equation}\label{m1talpha}
			\mathcal{M}^{\alpha}(t)=\int_{0}^{\infty}\mathcal{M}(y)\mathrm{Pr}(Y_{\alpha}(t)\in\mathrm{d}y).
		\end{equation}
		On substituting \eqref{m1t} in \eqref{m1talpha}, we obtain
		{\small\begin{equation}\label{m1ttt}
			\mathcal{M}^{\alpha}(t)=k\Big(\lambda\sum_{i=1}^{l}ic_{i}-\mu\Big)\int_{0}^{\infty}y\mathrm{Pr}(Y_{\alpha}(t)\in\mathrm{d}y)+k\mu\int_{0}^{\infty}\int_{0}^{y}p_{0}(u)\mathrm{d}u\mathrm{Pr}(Y_{\alpha}(t)\in\mathrm{d}y).
		\end{equation}}On taking the Laplace transform on both sides of \eqref{m1ttt} and by using \eqref{ltLnu}, we get
		\begin{align}
			\Tilde{\mathcal{M}}^{\alpha}(z)&=k\Big(\lambda\sum_{i=1}^{l}ic_{i}-\mu\Big)z^{\alpha-1}\int_{0}^{\infty}ye^{-yz^{\alpha}}\mathrm{d}y+k\mu z^{\alpha-1}\int_{0}^{\infty}\int_{0}^{y}p_{0}(u)e^{-yz^{\alpha}}\mathrm{d}u\mathrm{d}y\nonumber\\
			&=k\Big(\lambda\sum_{i=1}^{l}ic_{i}-\mu\Big)\frac{1}{z^{\alpha+1}}+k\mu z^{\alpha-1} \int_{0}^{\infty}\int_{u}^{\infty}p_{0}(u)e^{-yz^{\alpha}}\mathrm{d}y\mathrm{d}u\nonumber\\
			&=k\Big(\lambda\sum_{i=1}^{l}ic_{i}-\mu\Big)\frac{1}{z^{\alpha+1}}+k\mu z^{-1} \int_{0}^{\infty}p_{0}(u)e^{-uz^{\alpha}}\mathrm{d}u\nonumber\\
			&=k\Big(\lambda\sum_{i=1}^{l}ic_{i}-\mu\Big)\frac{1}{z^{\alpha+1}}+k\mu z^{-1}\sum_{h=1}^{\infty}\sum_{n=0}^{\infty}\sum_{\substack{\sum_{j=1}^{l}m_{j}=n\\m_{j}\geq0}}\frac{C_{h,n}^{0}(\textbf{m})}{(\lambda+k\mu+z^{\alpha})^{\gamma_{h,n}^{0}(\textbf{m})}}\nonumber\\
			&=k\Big(\lambda\sum_{i=1}^{l}ic_{i}-\mu\Big)\frac{1}{z^{\alpha+1}}+k\mu\sum_{h=1}^{\infty}\sum_{n=0}^{\infty}\sum_{\substack{\sum_{j=1}^{l}m_{j}=n\\m_{j}\geq0}}C_{h,n}^{0}(\textbf{m})\frac{z^{\alpha\gamma_{h,n}^{0}(\textbf{m})-\beta_{h,n}^{\mathcal{M}}(\textbf{m})}}{(\lambda+k\mu+z^{\alpha})^{\gamma^{0}_{h,n}(\textbf{m})}},\label{meanalphat}
		\end{align}
	where the penultimate step follows on using \eqref{int3}.
	Finally, on taking the inverse Laplace transform of \eqref{meanalphat} and using \eqref{ltm}, we get the required result.
	\end{proof}
	\begin{remark}
		For $l=1$, the mean queue length in \eqref{malpahat} reduces to that of the fractional Erlang queue (see Ascione {\it{et al.}} (2020), Eq. (47)).
	\end{remark}
	For an integrable function $f:[0,a]\rightarrow\mathbb{R}$, the Riemann-Liouville fractional integral is defined as (see Kilbas {\it{et al.}} (2006))
	\begin{equation}\label{fint}
		\mathcal{I}^{\alpha}_{t}f(t)\coloneqq\frac{1}{\Gamma(\alpha)}\int_{0}^{t}(t-y)^{\alpha-1}f(y)\mathrm{d}y.
	\end{equation}
	\begin{remark}
		On taking the Riemann fractional integral on both sides of \eqref{cpdtmalphat}, we get
		\begin{equation}\label{integrlp0t}
			\mathcal{M}^{\alpha}(t)=k\mu\mathcal{I}_{t}^{\alpha}p_{0}^{\alpha}(t)+k\Big(\lambda\sum_{i=1}^{l}ic_{i}-\mu\Big)\frac{t^{\alpha}}{\Gamma({\alpha+1})}.
		\end{equation}
		On comparing \eqref{malpahat} with \eqref{integrlp0t}, we get the Riemann integral of zero state probability of $\{\mathcal{Q}^{\alpha}(t)\}_{t\geq0}$
		\begin{equation*}
			\mathcal{I}_{t}^{\alpha}p_{0}^{\alpha}(t)=\sum_{h=1}^{\infty}\sum_{n=0}^{\infty}\sum_{\substack{\sum_{j=1}^{l}m_{j}=n\\m_{j}\geq0}}C_{h,n}^{0}(\textbf{m})t^{\beta_{h,n}^{\mathcal{M}}(\textbf{m})-1}E_{\alpha,\beta_{h,n}^{\mathcal{M}}(\textbf{m})}^{\gamma_{h,n}^{0}(\textbf{m})}(-(\lambda+k\mu) t^{\alpha}).
		\end{equation*}
	\end{remark}

	\section*{Acknowledgement}
	The
	research of first author was supported by a UGC fellowship, NTA reference no. 231610158041, Govt. of India.
	
\end{document}